\documentclass[11pt]{amsart}
\usepackage{amsmath,amsfonts,amssymb,amsthm}
\usepackage{mathtools}
\usepackage[a4paper, total={8.5in, 8.5in}]{geometry}
\usepackage[table]{xcolor}
\usepackage{enumitem}

\usepackage[utf8]{inputenc}
\usepackage[bottom]{footmisc}
\usepackage{here}

\usepackage{graphicx}

\newtheorem{theorem}{Theorem}[section]
\newtheorem{lemma}[theorem]{Lemma}
\newtheorem{prop}[theorem]{Proposition}

\newtheorem{conjecture}[theorem]{Conjecture}

\theoremstyle{definition}

\newtheorem{remark}[theorem]{Remark}

\makeatletter
  
  \@addtoreset{equation}{section}
   \makeatother
   
\begin{document}

\title[Multiple zeta-functions]{On the behavior of multiple zeta-functions
with identical arguments on the real line}

\author{Kohji Matsumoto}
\address{K. Matsumoto: Graduate School of Mathematics, Nagoya University, Chikusa-ku, Nagoya 464-8602, Japan}
\email{kohjimat@math.nagoya-u.ac.jp}

\author{Toshiki Matsusaka}
\address{T. Matsusaka: Institute for Advanced Research, Nagoya University, Chikusa-ku, Nagoya 464-8602, Japan}
\email{matsusaka.toshiki@math.nagoya-u.ac.jp}

\author{Ilija Tanackov}
\address{I. Tanackov: Faculty of Technical Sciences, University of Novi Sad,
Trg Dositeja Obradovi{\'c}a 6, 21000 Novi Sad, Serbia}
\email{ilijat@uns.ac.rs}

\keywords{multiple zeta-function, real zeros, asymptotic behavior, Newton's identities}
\subjclass[2010]{Primary 11M32, Secondary 11B83, 11M35}
\thanks{
Research of the first author is
supported by Grants-in-Aid for Science Research no. 18H01111, JSPS, 
that of the second author is by JP20K14292, JSPS,
and that of the
third author is by Ministry of Science and Technological Development of Serbia 
no. TR 36012.}

\begin{abstract}
We study the behavior of $r$-fold zeta-functions of Euler-Zagier type with identical
arguments $\zeta_r(s,s,\ldots,s)$ on the real line.    Our basic tool is an
``infinite'' version of Newton's classical identities.    We carry out numerical
computations, and draw graphs of $\zeta_r(s,s,\ldots,s)$ for real $s$, for several
small values of $r$.    Those graphs suggest various properties of 
$\zeta_r(s,s,\ldots,s)$, some of which we prove rigorously.    
When $s\in [0,1]$, we show that
$\zeta_r(s,s,\ldots,s)$ has $r$
asymptotes at $\Re s=1/k$ ($1\leq k\leq r$), and
determine the asymptotic behavior of $\zeta_r(s,s,\ldots,s)$ close to those asymptotes.  
Numerical computations establish the existence of several real zeros for $2\leq r\leq 10$
(in which only the case $r=2$ was previously known). 
Based on those computations, we raise a conjecture on the number of zeros for general
$r$, and gives a formula for calculating the number of zeros.
We also consider the behavior of $\zeta_r(s,s,\ldots,s)$ outside the interval
$[0,1]$.    We prove asymptotic formulas for $\zeta_r(-k,-k,\ldots,-k)$, where $k$ takes odd
positive integer values and tends to $+\infty$.
Moreover, on the number of real zeros of $\zeta_r(s,s,\ldots,s)$,
we prove that there are exactly $(r-1)$ real zeros on the interrval $(-2n,-2(n-1))$
for any $n\geq 2$.
\end{abstract}

\maketitle

\section{Introduction}

The Euler-Zagier multiple zeta-function
\begin{align}\label{EZ_def}
\zeta_r (s_1,s_2,...,s_r)= \sum_{1 \leq m_1 <m_2<...<m_r } \frac{1}{m_1^{s_1}} \cdot\frac{1}{m_2^{s_2}} \cdots\frac{1}{m_r^{s_r}}, 
\end{align}
where $s_1,\ldots,s_r$ are complex variables, has been studied extensively in
recent decades.   The series \eqref{EZ_def} is convergent absolutely when $\Re s_j$
($1\leq j\leq r$) are sufficiently large, but it can be continued meromorphically to
the whole space $\mathbb{C}^r$ (see, for example, \cite{AET01}).
Analytic properties of $\zeta (s_1,s_2,...,s_r)$ have been studied in a lot of papers,
among which we mention here a numerical study on $\zeta_2(s_1,s_2)$ by the
first-named author and M. Sh{\=o}ji \cite{MatSho14} \cite{MatSho20}.    In \cite{MatSho14}
the case $s_1=s_2=s$ was treated, and the general two-variable case was discussed in
\cite{MatSho20}.
In particular, in \cite{MatSho14} it has been shown that the distribution of the zeros of
$\zeta_2(s,s)$ is not similar to that of the Riemann zeta-function $\zeta(s)=\zeta_1(s)$
(especially the Riemann hypothesis does not hold), but rather, has a resemblance to
the distribution of the zeros of Hurwitz zeta-functions
$\zeta(s,\alpha)=\sum_{m=0}^{\infty}(m+\alpha)^{-s}$ ($0<\alpha\leq 1$).   

It is desirable to generalize the study \cite{MatSho14} \cite{MatSho20} to the
general $r$-fold case.     It is natural to consider first the case when all variables
are identical: $s_1=\cdots=s_r=s$.    We write
$$
\zeta_r(s)=\zeta_r(s,s,\ldots,s).
$$
In \cite{MatSho14}, two topics were considered; the distribution of zeros of
$\zeta_2(s)$ in the complex plane $\mathbb{C}$, and the behavior of $\zeta_2(s)$ on the
real axis $\mathbb{R}$.    The aim of the present paper is to study the behavior of
$\zeta_r(s)$ ($r\geq 2$) on $\mathbb{R}$.    The matter on the zeros of $\zeta_r(s)$
in $\mathbb{C}$ will be treated in our subsequent paper.

Our investigation is based on numerical computations on the behavior of
$\zeta_r(s)$ for $s\in\mathbb{R}$.   
We draw the graphs of $\zeta_r(s)$ for $2\leq r\leq 10$, which
suggest various properties of 
$\zeta_r(s,s,\ldots,s)$, some of which we prove rigorously.   
Our main results are as follows.
 
When $s\in [0,1]$, we show that
$\zeta_r(s,s,\ldots,s)$ has $r$
asymptotes at $\Re s=1/k$ ($1\leq k\leq r$), and
determine the asymptotic behavior of $\zeta_r(s,s,\ldots,s)$ close to those asymptotes.  
Numerical computations establish the existence of several real zeros for $2\leq r\leq 10$
(in which only the case $r=2$ was previously known in \cite{MatSho14}). 
Based on those computations, we raise a conjecture on the number of zeros for general
$r$, and gives a formula for calculating the number of zeros.

We also consider the behavior of $\zeta_r(s,s,\ldots,s)$ outside the interval
$[0,1]$.    When $s>1$, we investigate the behavior of $\zeta_r(s,s,\ldots,s)$
as $s\to \infty$, and as $r\to\infty$.
When $s<0$, we prove asymptotic formulas for $\zeta_r(-k,-k,\ldots,-k)$, where $k$ takes odd
positive integer values and tends to $+\infty$.
Moreover, on the number of real zeros of $\zeta_r(s,s,\ldots,s)$,
we prove that there are exactly $(r-1)$ real zeros on the interval $(-2n,-2(n-1))$
for any $n\geq 2$.


As we mentioned above, there is some similarity between the behavior of $\zeta_2(s)$ 
and that of Hurwitz zeta-functions.   Such similarity can also be expected for
$\zeta_r(s)$, $r\geq 3$.     Therefore the study on the real zeros of Hurwitz
zeta-functions is suggestive in our research.

Recently there has been big progress on the study of real zeros of Hurwitz
zeta-functions (see Schipani \cite{Sch11}, Nakamura \cite{Nak16a} \cite{Nak16b},
Matsusaka \cite{Mat18}, and Endo and Suzuki \cite{EndSuz19}). 
It is to be mentioned that the idea included in those articles
was already applied by Nakamura himself \cite{Nak16b} to a variant of $\zeta_2(s)$ of
Hurwitz-Lerch type, and by Sakurai \cite{SakPre} to Barnes double zeta-functions.
Now we may say that our result on the number of zeros of $\zeta_r(s,s,\ldots,s)$
on the interval 
$(-2n,-2(n-1))$ ($n\geq 2$) is an analogue of the result proved in \cite{Mat18}.

\section{Newton's identities}\label{sec-2}

In this section we prepare the basic identities among multiple zeta-functions, based 
on the classical identities of
Newton, which we will use in our computations.    

Let $\mathbb{N}$ be the set of positive integers.
The polynomial of degree $n\in\mathbb{N}$, with roots $x_1,\ldots,x_n$ may be written as
\begin{align}
\prod_{m=1}^n(x-x_m)=\sum_{r=0}^n (-1)^r e_r x^{n-r},
\end{align}
where $e_r$ are symmetric polynomilas given by
\begin{align}
e_r=e_r(x_1,x_2,\ldots,x_n)=\sum_{1\leq m_1<m_2<\cdots<m_r\leq n}x_{m_1}x_{m_2}
\cdots x_{m_r}\qquad (1\leq r\leq n).
\end{align}
For example $e_1=x_1+x_2+\cdots+x_n$, $e_2=\sum_{1\leq i<j\leq n}x_i x_j$, etc., 
and we interpret that $e_0=1$.

Define the $r$-th power sum
$$
p_r=p_r(x_1,x_2,\ldots,x_n)=x_1^r+x_2^r+\cdots+x_n^r.
$$
Newton's identities are given by the following statement:
\begin{align}\label{Newton_id}
re_r(x_1,x_2,\ldots,x_n)=\sum_{j=1}^r (-1)^{j-1}e_{r-j}(x_1,x_2,\ldots,x_n)
p_j(x_1,x_2,\ldots,x_n),
\end{align}
where $r,n\in\mathbb{N}$ with $r\leq n$.
This is due to Sir Issac Newton.   Various proofs can be found in, for example,
Zolberger \cite{Zol84}, Kalman \cite{Kal00}, and Mukherjee and Bera \cite{MukBer19}.

Now we let $x_m=m^{-s}$, where $s$ is a complex variable, and put
\begin{align}\label{N_r_def}
N_r(s)=e_r(1^{-s},2^{-s},\ldots,n^{-s})=\sum_{1\leq m_1<m_2<\cdots<m_r\leq n}
(m_1 m_2\cdots m_r)^{-s} \qquad (1\leq r\leq n).
\end{align}
We take the limit $n\to\infty$.    When $\Re s>1$, the limit of the right-hand side
converges, and is equal to $\zeta_r(s)$.    Therefore,
\begin{align}
\lim_{n\to\infty}N_r(s)=\zeta_r(s) \qquad (\Re s>1).
\end{align}
Also we see that
\begin{align}
\lim_{n\to\infty}p_j(1^{-s},2^{-s},\ldots,n^{-s})=\sum_{m=1}^{\infty}m^{-js}=\zeta(js),
\end{align}
where $\zeta(s)=\zeta_1(s)$ is the Riemann zeta-function.
Therefore, taking the limit $n\to\infty$ of Newton's identity \eqref{Newton_id}, 
we obtain
\begin{align}\label{zeta_id}
r\zeta_r(s)=\sum_{j=1}^r (-1)^{j-1}\zeta_{r-j}(s)\zeta(js) \qquad (r\in\mathbb{N}),
\end{align}
where we understand that $\zeta_0=1$.
This identity is first valid for $\Re s>1$, but then, by the meromorphic 
continuation, it can be extended to any $s\in\mathbb{C}$.

This is not a new identity.
In fact, this is essentially the well-known harmonic product formula, and Kamano \cite{Kam06}
deduced \eqref{zeta_id} (in a little more generalized form) from the harmonic product
formula.    However our argument is different.

For several small values of $r$, this formula implies:
\begin{align}\label{2_id}
\zeta_2(s)=\frac{1}{2}\left\{\zeta(s)^2-\zeta(2s)\right\},
\end{align}
\begin{align}\label{3_id}
\zeta_3(s)=\frac{1}{3}\left\{\zeta_2(s)\zeta(s)-\zeta(s)\zeta(2s)+\zeta(3s)\right\},
\end{align}
\begin{align}\label{4_id}
\zeta_4(s)=\frac{1}{4}\left\{\zeta_3(s)\zeta(s)-\zeta_2(s)\zeta(2s)+\zeta(s)\zeta(3s)
-\zeta(4s)\right\}.
\end{align}
From these formulas, it is also possible to get expressions of $\zeta_r(s)$ only in
terms of the Riemann zeta-function.   Substituting \eqref{2_id} into \eqref{3_id}, we
obtain
\begin{align}\label{3_riemann}
\zeta_3(s)=\frac{1}{6}\left\{\zeta(s)^3-3\zeta(s)\zeta(2s)+2\zeta(3s)\right\}.
\end{align}
Similarly,
\begin{align}\label{4_riemann}
\zeta_4(s)=\frac{1}{24}\left\{\zeta(s)^4-6\zeta(s)^2\zeta(2s)+3\zeta(2s)^2
+8\zeta(s)\zeta(3s)-6\zeta(4s)\right\}.
\end{align}

Consider a formal infinite polynomial 
\begin{align}\label{inf_poly}
 \prod_{m=1}^{ \infty } (x- x_m)= \lim_{n\to \infty}  \left\{( x-x_1)(x- x_2)\cdots
 (x- x_n )\right\}.
 \end{align}
When $x_m=m^{-s}$, the right-hand side is equal to
\begin{align}
\lim_{n\to \infty}\left\{x^n-x^{n-1} \sum_{m_1=1}^{n} \frac{{1} }{m_1^s}+x^{n-2}  \sum_{1 \leq m_1<m_2\leq n} \frac{{1} }{(m_1m_2)^s} -x^{n-3}\sum_{1 \leq m_1<m_2<m_3\leq n} \frac{{1} }{(m_1m_2m_3)^s}+...+\frac{{(-1)^n} }{{(n!)^s}} \right\},
\end{align}
whose each coefficient tends to $\zeta_r(s)$ as $n\to\infty$.
It should be noted that the above expansion is based on the rules of Vieta (F. Vi{\`e}te)
for the infinite polynomial \eqref{inf_poly}.
This observation shows that our fundamental formula \eqref{zeta_id} may be
formally regarded as Newton's identities for the infinite polynomial
\eqref{inf_poly}.

The above argument was inspired by the third-named author's
``new-nacci" method for solving polynomial roots.    This method is based on the convergence of successive
Fibonacci-type sequences, 
which are Newton's identities (see Tanackov et al.~\cite{Tan20}).
The new-nacci method also gives the relation
\begin{align}\label{plus_inf}
\lim_{k\to\infty}\frac{\zeta_r(k+1)}{\zeta_r(k)}=\frac{1}{r!},
\end{align}
which is useful when we consider the behavior of
$\zeta_r(s)$ when $s\to +\infty$.
A proof of this relation will be given in Section \ref{sec-5}.


\section{The double and the triple zeta-functions}\label{sec-3}

Before going into the discussion of general $r$-fold situation, in this section
we observe the behavior of $\zeta_r(s)$ on the real line for $r=2$ and 3.

\subsection{The double case}

The analytic properties of the double zeta-function are well studied (see Matsumoto \cite{Mat03} \cite{Mat04}, Kiuchi et al. \cite{KiuTan06} \cite{KTZ11}, Matsumoto and
Tsumura \cite{MatTsu15}, etc.).
The double zeta-function with identical arguments $\zeta_2(s)$, $s\in\mathbb{R}$,
can be computed by \eqref{2_id} (see Fig \ref{Fig1}).

\begin{figure}[h]
\centering
\includegraphics{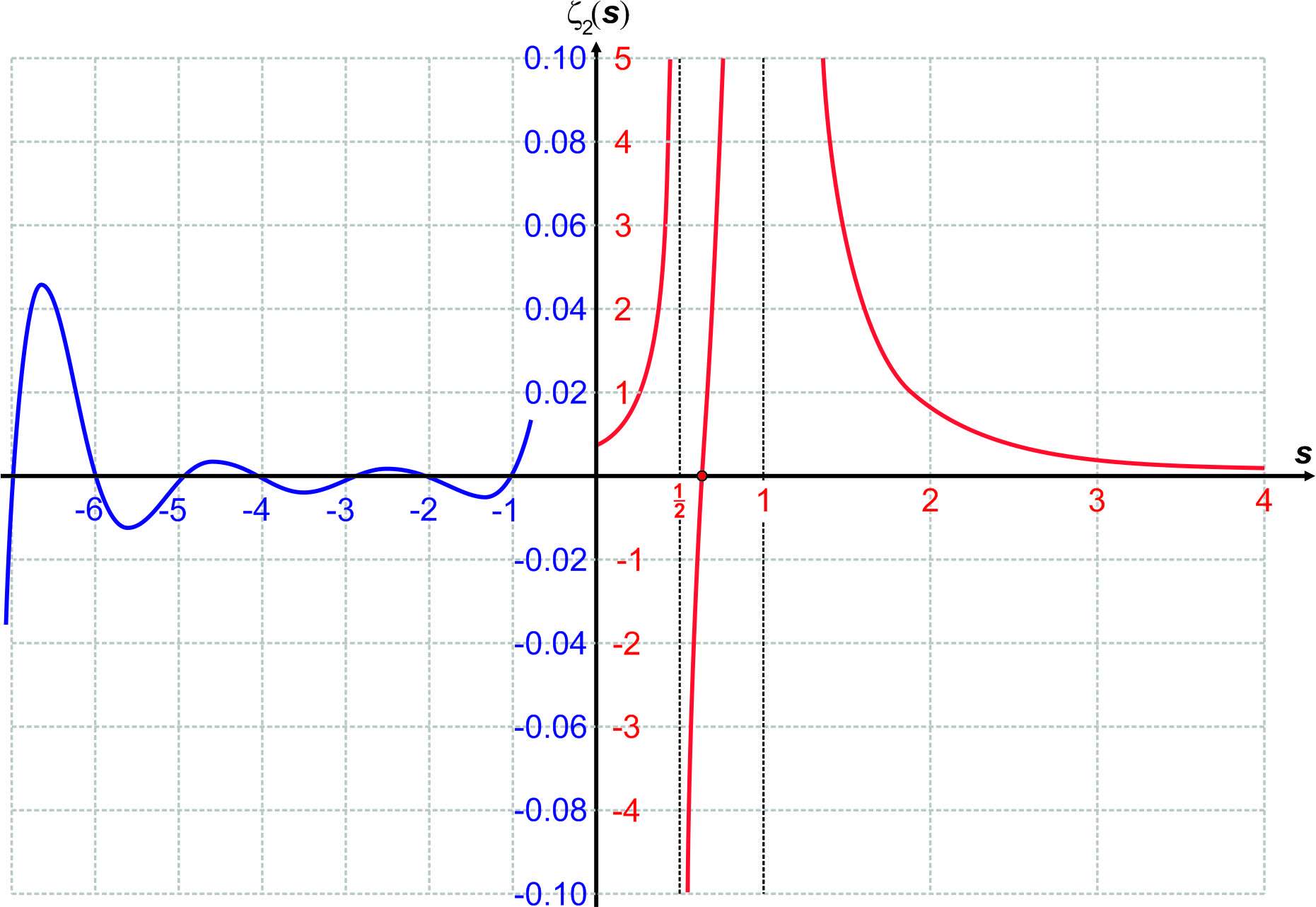}
\caption{The double zeta-function (Note that the vertical scale in the negative half-plane
is different from that in the positive half-plane)}
\label{Fig1}
\end{figure}

From \eqref{2_id} it is clear that all trivial zeros of the Riemann zeta-function $s=-2,-4,-6,\ldots$ are trivial zeros of the double zeta-function. 
The values of $\zeta_2(s)$ at $s=-1,-3,-5,\ldots$ were already computed by 
Kamano \cite{Kam06}.
For example, since $\zeta_1(-1) =-12^{-1}$, we find
\begin{equation}\label{eq:4.2.}
\zeta_2(-1)=\frac{1}{2}(\zeta_1^2(-1)-\zeta_1(-2))
=\frac{1}{2}\left(\left(-\frac{1}{12}\right)^2-0\right)=\frac{1}{288}.
\end{equation}

Let $\ell\in\mathbb{N}$.    For any $\ell$, there should be at least one
``inter-trivial'' zero (ITZ for brevity) between two trivial zeros $s=-2\ell$ and $s=-2(\ell+1)$.
The proof is as follows.   Since all trivial zeros of $\zeta(s)$ are simple, the value of
$\zeta'(-2\ell)$ is positive for even $\ell$ and negative for odd $\ell$.
From \eqref{2_id} we have 
$\zeta_2^{\prime}(s)=\zeta(s)\zeta'(s)-\zeta'(2s)$, hence 
$\zeta_2^{\prime}(-2\ell)=-\zeta'(-4\ell)$, which is negative for all $\ell$.
Therefore $\zeta_2(s)$ is decreasing at $s=-2\ell$, which implies the claim.


The graph of the double zeta-function has two vertical asymptotes $\Re s=1$ and
$\Re s=1/2$ in the positive half-plane.    The first one is caused by the factor
$\zeta(s)$, while the second one by $\zeta(2s)$, respectively.
One zero \(\zeta_2(0.6268175...)=0\) appears between these two vertical asymptotes.   We call this type of zero an ``inter-asymptotic'' zero, and write IAZ for brevity.

When $s\to +\infty$, the graph is smoothly going down to 0 because of \eqref{plus_inf},
while when $s\to -\infty$ the graph is strongly oscillated; indeed, 
$\zeta_2(-2\ell-1)\to +\infty$ as $\ell\to\infty$.

All the 
stated values, the graph, the location of zeros for $\zeta_2(s)$ mentioned above were analyzed in detail by the first-named author and Sh{\=o}ji \cite{MatSho14}.

\subsection{The triple case}

Unlike the double zeta, the analytical properties of the triple zeta-function have been more modestly
investigated (see Kiuchi and Tanigawa \cite{KiuTan08}).
The values of the triple zeta-function have been discussed for
different arguments of positive integer values (Markett \cite{Mar94}, 
Hoffman and Moen \cite{HofMoe96}, and Machide \cite{Mac13}).
%

We compute $\zeta_3(s)$ for $s\in\mathbb{R}$ by using \eqref{3_id}, \eqref{3_riemann}
(see Fig \ref{Fig3}).
For example, $\zeta_3(-1)=139/51840$.

\begin{figure}[h]
\centering
\includegraphics{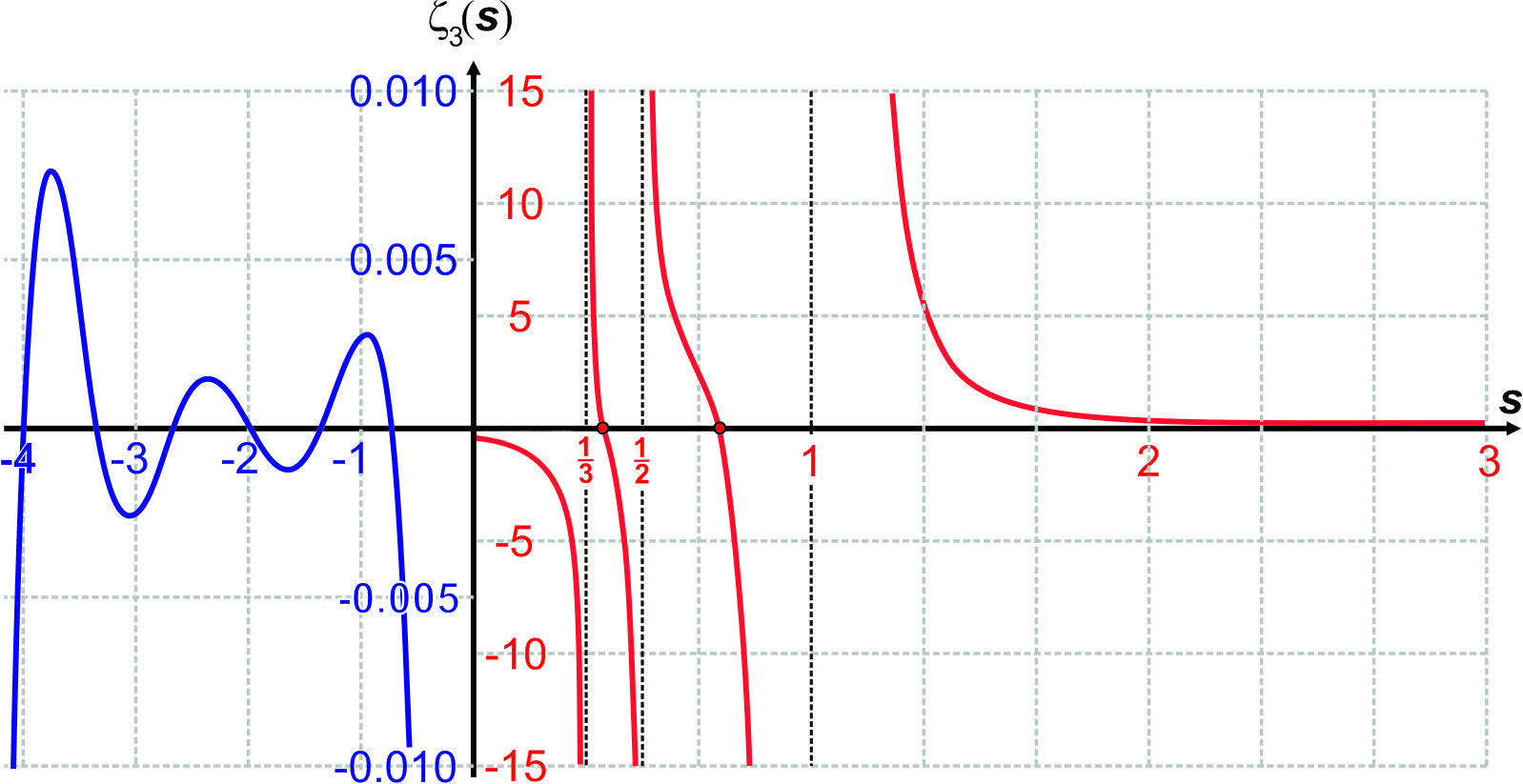}
\caption{The triple zeta-function (Note that the vertical scale in the negative half-plane
is different from that in the positive half-plane)}
\label{Fig3}
\end{figure}

All trivial zeros of the Riemann zeta-function are zeros of $\zeta_3(s)$.    
It seems from the graph that, on the interval between consecutive trivial zeros of the Riemann zeta-function, $\zeta_3(s)$ has two ITZs.
In the positive half-plane, the graph of the triple zeta-function has three vertical asymptotes: $\Re s=1, 1/2$ and $1/3$.
Vertical asymptotes $\Re s=1$ and $1/2$ are inherited from the double zeta-function,
while the new vertical asymptote $\Re s=1/3$ is coming from the factor $\zeta(3s)$ 
in \eqref{3_id}, \eqref{3_riemann}.
There are two IAZs between the vertical asymptotes \(\zeta_3(+0.385782)=0\) and
\(\zeta_3(+0.724902)=0\).
The behavior of $\zeta_3(s)$ when $s\to +\infty$ can be explained by \eqref{plus_inf}.


\section{Multiple zeta-functions on the interval $[0,1]$}

Now we proceed to the study of general $r$-fold multiple zeta-functions.
We first investigate the behavior in the interval $[0,1]$, because this is
the most intriguing and dynamic interval.
 
\subsection{The cases $\mathbf{{\it r}=4, 5}$ and 6}

In this subsection, we present the graphs of $\zeta_r(s)$, $s\in [0,1]$, for
$r=4,5$ and $6$.
 
The quadruple zeta-function was studied, for example, by Machide \cite{Mac19}.
The theory of the fourth power mean of the Riemann zeta-function
(Motohashi \cite{Mot93}, Ivi{\'c} and Motohashi \cite{IviMot95}) is somewhat relevant.

The quadruple zeta-function $\zeta_4(s)$ has four 
asymptotes: $\Re s=1, 1/2, 1/3$ and $1/4$, and
has four IAZs in the interval \( s\in [0,1]\) 
(see Fig \ref{Fig6}):

\begin{itemize}[noitemsep]
    \item One IAZ $\in (1/4,1/3)$:\;\(\zeta_4(0,27886...)\approx 0.\)
    
    \item One IAZ $\in (1/3,1/2)$:\;\(\zeta_4(0,387072...)\approx 0 .\)
    \item Two IAZs $\in (1/2,1)$: \(\zeta_4(0,571348...) \approx 0 \) and \(\zeta_4(0,783444...)\approx 0. \)

\end{itemize}

The quadruple zeta function \(\zeta_4(s)\) has one minimum \(\zeta_4(0,693658...) \approx -4,0699572... \) between vertical
asymptotes $1/2$ and 1.

There are two new features here:

(i) When $s\to 1/2$, the value $\zeta_4(s)$ tends to $+\infty$ for the both of the limits
$s\to 1/2+0$ and $s\to 1/2-0$.    The reason is that the pole of $\zeta_4(s)$ at $s=1/2$
is of order 2, because of the term $3\zeta(2s)^2$ in \eqref{4_riemann}.

(ii) In the interval $(1/2,1)$ there are two zeros.


\begin{figure}[h]
\centering
\includegraphics{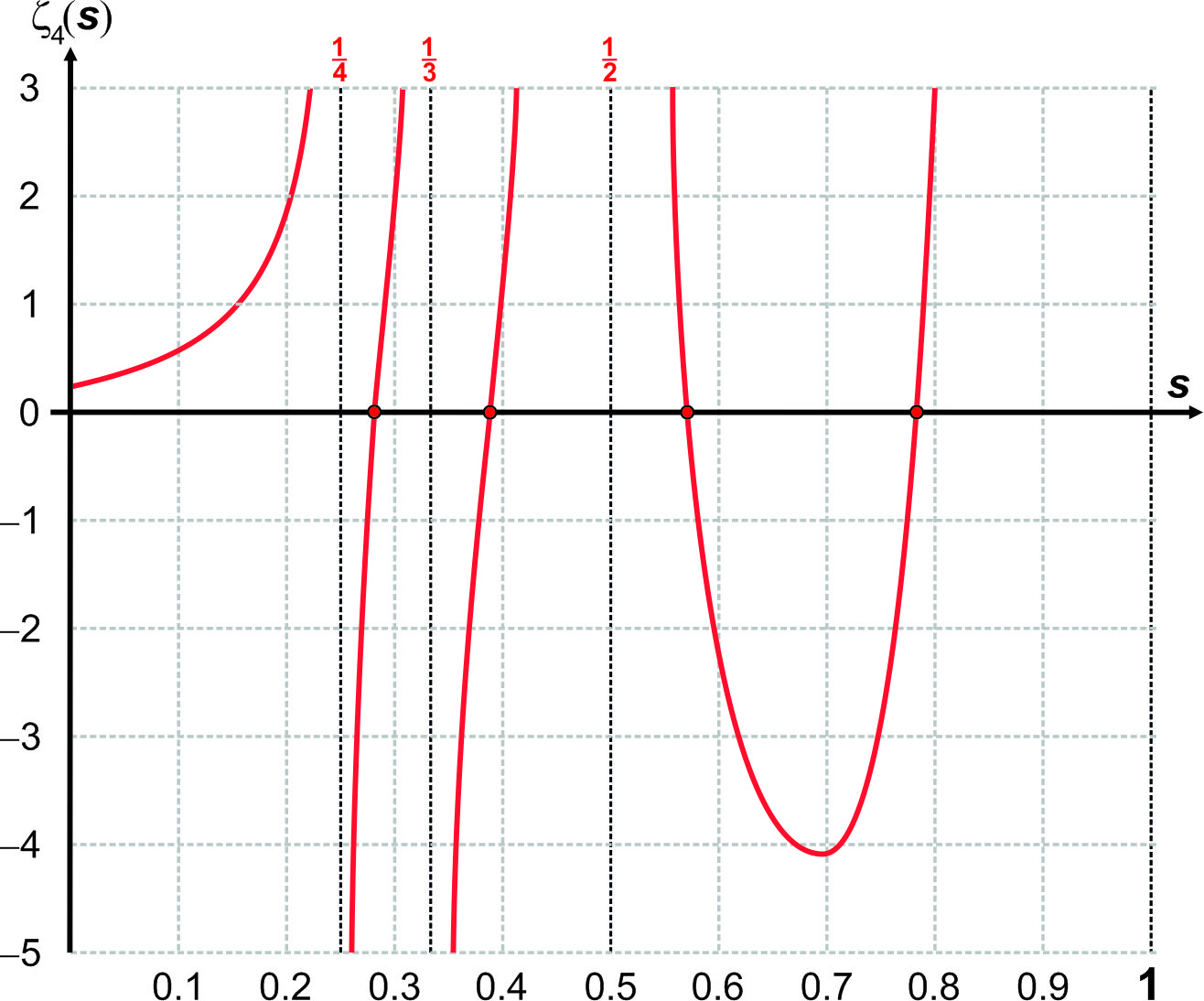}
\caption{The quadruple zeta-function for $s\in [0,1]$}
\label{Fig6}
\end{figure}

Next,
the five-fold zeta-function $\zeta_5(s)$ has five asymptotes: $\Re s=1, 1/2, 1/3, 1/4$ 
and $1/5$, and has five IAZs in \( s\in [0,1]\) (Fig \ref{Fig7}):

\begin{itemize}[noitemsep]
    \item One IAZ $\in (1/5,1/4)$: \(\zeta_5(0,218315...)\approx 0.\)
    
    \item One IAZ $\in (1/4,1/3)$: \(\zeta_5(0,278346...)\approx 0.\)
    \item One IAZ $\in (1/3,1/2)$: \(\zeta_5(0,423505...) \approx 0. \) 
    \item Two IAZs $\in (1/2,1)$: \(\zeta_5(0,643861...)\approx 0 \) and \(\zeta_5(0,881698...)\approx 0.\)
 \end{itemize}
Also \(\zeta_5(s)\) has one maximum \(\zeta_5(0,776027...)\approx  +6,003808... \) between vertical asymptotes $1/2$ and 1.
We see that $\zeta_5(s)\to -\infty$ as $s\to 1/2\pm0$.    

\begin{figure}[h]
\centering
\includegraphics{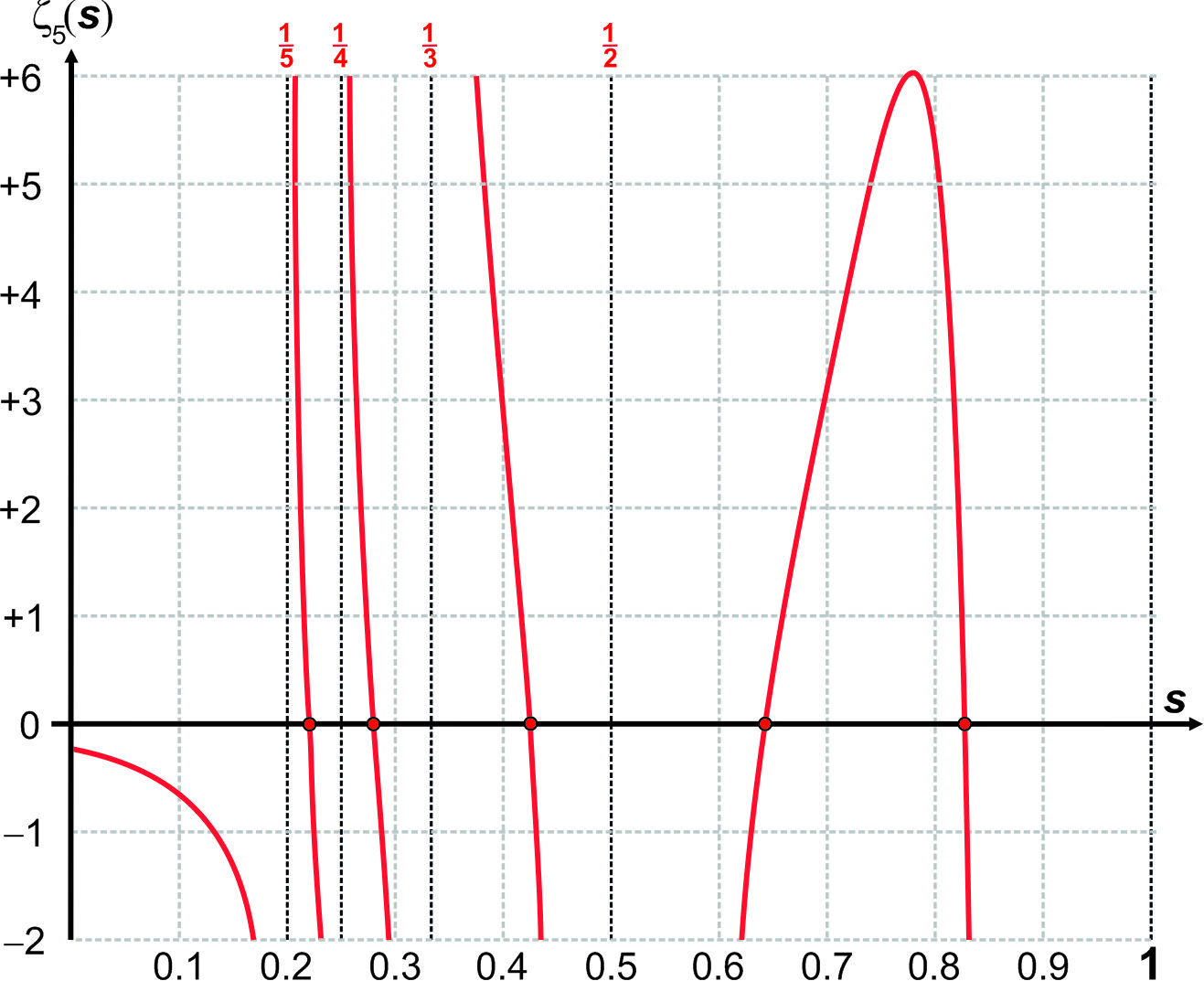}
\caption{The five-fold zeta-function for $s\in [0,1]$}
\label{Fig7}
\end{figure}

The six-fold zeta-function has six asymptotes: $\Re s=1, 1/2, 1/3, 1/4, 1/5$ 
and  $1/6$, and has eight IAZs in \( s\in [0,1]\) (Fig \ref{Fig8}): 

\begin{itemize}[noitemsep]
    \item One IAZ $\in (1/6,1/5)$: \(\zeta_6(0,179347...)\approx 0.\)
    
    \item One IAZ $\in (1/5,1/4)$: \(\zeta_6(0,217682...)\approx 0.\)
    \item  One IAZ $\in (1/4,1/3)$: \(\zeta_6(0,279817...) \approx0 \).
    \item Two IAZs $\in (1/3,1/2)$: \(\zeta_6(0,362716...)\approx 0 \) and \(\zeta_6(0,419205...)\approx 0. \)
    
    \item Three IAZs $\in (1/2,1)$: \(\zeta_6(0,549629...)\approx 0 \), \(\zeta_6(0,696745...)\approx 0 \) and
\(\zeta_6(0,848546...)\approx 0.\) 
\end{itemize}

The six-fold zeta-function \(\zeta_6(s)\) has two minimums and one maximum in \( s\in [0,1]\) :

\begin{itemize}[noitemsep]
    \item One minimum between vertical asymptotes $1/3$ and $1/2$: \(\zeta_6(0,386562...)\approx  -2,462682...\)
    \item One maximum \(\zeta_6(0,578067...)\approx  +5,283455...\)and one minimum \(\zeta_6(0,818945...)\approx  -10,900018...\)
between vertical asymptotes $1/2$ and 1.
\end{itemize}

There again appear two new features.    The case $r=6$ of \eqref{zeta_id} gives
\begin{align}
\zeta_6(s)=\frac{1}{6}\left\{\zeta_5(s)\zeta(s)-\zeta_4(s)\zeta(2s)+\zeta_3(s)\zeta(3s)
-\zeta_2(s)\zeta(4s)+\zeta(s)\zeta(5s)-\zeta(6s)\right\},
\end{align}
from which we can see that the orders of the poles of $\zeta_6(s)$ at $s=1/3$ and
$s=1/2$ are 2 and 3, respectively.   This gives the behavior of $\zeta_6(s)$ around
the asymptotes $1/3$ and $1/2$, indicated in the figure.
Moreover, the interval $(1/2,1)$ now includes three IAZ, and the interval $(1/3,1/2)$
also includes more than one IAZ.

\begin{figure}[h]
\centering
\includegraphics{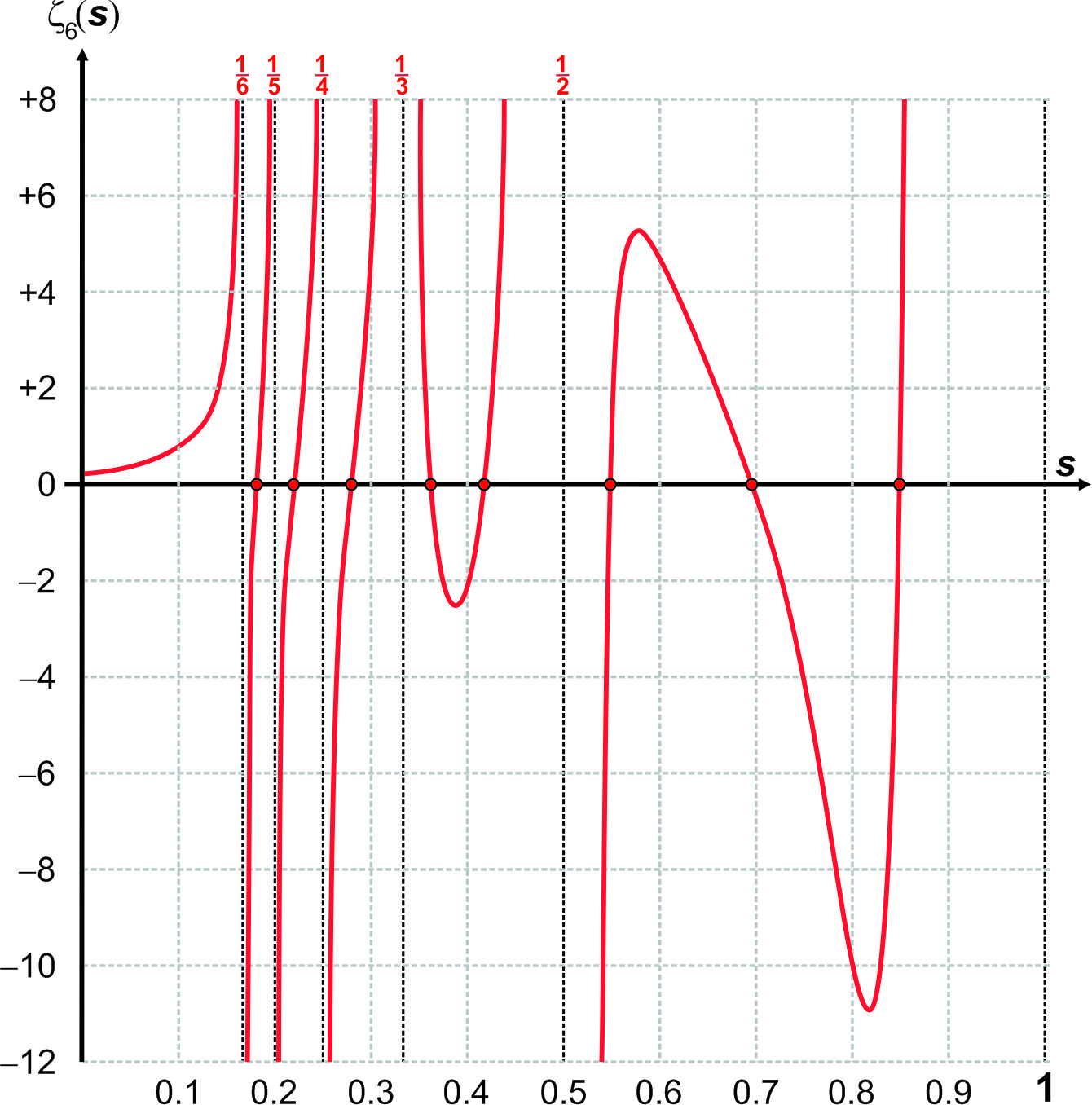}
\caption{The six-fold zeta-function for $s\in [0,1]$}
\label{Fig8}
\end{figure}
 
\subsection{Theorems and a conjecture}
 
The above numerical data suggests several properties of $\zeta_r(s)$, some of which we will
prove here.   The first result is:

\begin{theorem}\label{Th<1/r}
On the interval $0\leq s< 1/r$, the value of $\zeta_r(s)$ is positive for even $r$,
and negative for odd $r$.    In particular, there is no zero of $\zeta_r(s)$ in
this interval.
\end{theorem}    

\begin{proof}
When $r=1$, it is well-known that $\zeta(s)<0$ for $0\leq s<1$.   Then, from
\eqref{2_id} it is clear that $\zeta_2(s)>0$ for $0\leq s<1/2$.

In general, we prove the assertion by induction.
From \eqref{zeta_id} we have
\begin{align}\label{zeta_id2}
\zeta_r(s)=\frac{1}{r}\sum_{j=1}^{r-1}(-1)^{j-1}\zeta_{r-j}(s)\zeta(js)+
\frac{1}{r}(-1)^{r-1}\zeta(rs).
\end{align}
Let $0\leq s< 1/r$.    Then we have $\zeta(rs)<0$, hence the signature of the term 
$r^{-1}(-1)^{r-1}\zeta(rs)$ is given by $(-1)^r$.
Using the induction assumption, we have that $\zeta_{r-j}(s)$ is non-zero and its 
signature is
given by $(-1)^{r-j}$, so the signature of the term $(-1)^{j-1}\zeta_{r-j}(s)\zeta(js)$
is given by $(-1)^{(j-1)+(r-j)+1}=(-1)^r$.
Therefore all terms on the right-hand side of \eqref{zeta_id2} have the same signature
$(-1)^r$, and so the signature of $\zeta_r(s)$.
\end{proof}

Using Theorem \ref{Th<1/r}, we can prove the following theorem on
the asymptotic behavior of $\zeta_r(s)$ near the asymptotes.

\begin{theorem}\label{Th-asymp}
{\rm (i)} The function $\zeta_r(s)$ has poles only at $s=1/k$ $(1\leq k\leq r)$ of 
order $[r/k]$, where $[x]$ denotes the integer part of $x$.

{\rm (ii)} As $s\to 1/k$, the asymptotic behavior of $\zeta_r(s)$ is given by
\begin{align}\label{Th_formula}
\zeta_r(s)\sim C_r(k)(ks-1)^{-[r/k]},
\end{align}
where $C_r(k)$ is a non-zero real constant, whose singature coincides with the signature of
$(-1)^{r+[r/k]}$.    
\end{theorem}

\begin{proof}
When $r=1$, $\zeta_1(s)=\zeta(s)$ is the Riemann zeta-function, which has the only
pole at $s=1$ and this pole is simple with residue 1.    Therefore the theorem
for $r=1$ clearly holds. 

Now assume $r\geq 2$, and we prove the theorem by induction.
We use \eqref{zeta_id2} again.   On the right-hand side of \eqref{zeta_id2},
the factor $\zeta(js)$ has the only pole at $s=1/j$, while by induction assumption,
the poles of $\zeta_{r-j}(s)$ are at $s=1,1/2,\ldots,1/(r-j)$.
Therefore, for any fixed $k$ ($1\leq k\leq r$), the factors on the right-hand side of
\eqref{zeta_id2}, which are singular at $s=1/k$, are $\zeta(ks)$ and
$\zeta_{r-j}(s)$ with $j\leq r-k$.

Consider the case $k=r$.    The singularity at $s=1/r$ appears only in the last
term $r^{-1}(-1)^{r-1}\zeta(rs)$.    Therefore
\begin{align}
\zeta_r(s)\sim \frac{1}{r}(-1)^{r-1}\frac{1}{rs-1}
\end{align}
as $s \to 1/r$,
which implies \eqref{Th_formula} for $k=r$, with $C_r(r)=r^{-1}(-1)^{r-1}$.

Next, let $1\leq k\leq r-1$.
The asymptotic behavior of the singular factors at $s=1/k$ on the right-hand side of
\eqref{zeta_id2} can be written down by induction
assumption.    
For $j\leq r-k$ with $j\neq k$, we have
$$
\zeta_{r-j}(s)\zeta(js)\sim \zeta(j/k)C_{r-j}(k)(ks-1)^{-[(r-j)/k]}
$$ 
as $s\to 1/k$.
As for the term $\zeta_{r-k}(s)\zeta(ks)$, the factor $\zeta(ks)$ is always singular at 
$s=1/k$, while $\zeta_{r-k}(s)$ is singular only when $k\leq r-k$, that is $k\leq r/2$.
Therefore we may write
\begin{align}\label{zeta_asymp}
\zeta_r(s)\sim A(s)+B(s)
\end{align}
as $s\to 1/k$, where
\begin{align}\label{A_def}
A(s)=\frac{1}{r}\sum_{\stackrel{1\leq j\leq r-k}{j\neq k}}(-1)^{j-1}
\zeta\left(\frac{j}{k}\right)C_{r-j}(k)(ks-1)^{-[(r-j)/k]}
\end{align}
and
\begin{align}
B(s)=\left\{
\begin{array}{lll}
\displaystyle{\frac{1}{r}(-1)^{k-1}C_{r-k}(k)
(ks-1)^{-[(r-k)/k]}\frac{1}{ks-1}}  & {\rm if} & k\leq r/2,\\
\displaystyle{\frac{1}{r}(-1)^{k-1}\zeta_{r-k}\left(\frac{1}{k}\right)
\frac{1}{ks-1}}  & {\rm if}  &  k>r/2,
\end{array}
\right.
\end{align}
say.    Since
$(ks-1)^{-[(r-k)/k]-1}=(ks-1)^{-[r/k]}$, and if $k>r/2$ then $[r/k]=1$, we can
unify the above expression of $B(s)$ as
$$
B(s)
=\frac{1}{r}(-1)^{k-1}D_{r-k}(k)(ks-1)^{-[r/k]},
$$
where 
\begin{align}
D_{r-k}(k)=\left\{
\begin{array}{lll}
C_{r-k}(k) & {\rm if} & k\leq r/2,\\
\zeta_{r-k}(1/k) & {\rm if} &  k>r/2.
\end{array}
\right.
\end{align}
Since $[(r-j)/k]\leq [r/k]$ for any $j\geq 1$, the main contribution on the right-hand
side of \eqref{zeta_asymp} would be coming from $B(s)$, and the part of $A(s)$
consisting of only $j$ satisfying $[(r-j)/k]=[r/k]$.    That is, it would be that
\begin{align}\label{zeta_sim}
\zeta_r(s)\sim C_r(k)(ks-1)^{-[r/k]},
\end{align}
where
\begin{align}\label{C_r_k_def}
C_r(k)= \frac{1}{r}\left\{{\sum_j}^* (-1)^{j-1}\zeta\left(\frac{j}{k}\right)C_{r-j}(k)
+(-1)^{k-1}D_{r-k}(k)\right\},
\end{align}
and here, the symbol $\sum^*$ stands for the summation on $j$ satisfying
$j\leq r-k$, $j\neq k$ and $[(r-j)/k]=[r/k]$. 

To establish \eqref{zeta_sim} rigorously, it is necessary to check that
$C_r(k)\neq 0$.    This can be seen by observing the signature.    
By induction assumption, $C_{r-j}(k)$ is non-zero, and the signature of $C_{r-j}(k)$ coincides with the 
signature of $(-1)^{r-j+[(r-j)/k]}$.
Therefore the signature of $(-1)^{k-1}C_{r-k}(k)$ is 
$$(-1)^{k-1+(r-k+[(r-k)/k])}=(-1)^{r+[r/k]},$$
while the signature of $(-1)^{j-1}\zeta(j/k)C_{r-j}(k)$ is
$$
(-1)^{(j-1)+1+(r-j+[(r-j)/k]) }=(-1)^{r+[r/k]}
$$
for any $j$ which appears in the sum $\sum^*$
(because of the condition on $\sum^*$, and the fact $\zeta(j/k)<0$).
Moreover, if $k>r/2$, then $1/k<1/(r-k)$ so, by Theorem \ref{Th<1/r}, the signature 
of $(-1)^{k-1}\zeta_{r-k}(1/k)$ is $(-1)^{(k-1)+(r-k)}=(-1)^{r-1}$, and this is 
further equal to 
$(-1)^{r+[r/k]}$ because now $[r/k]=1$.   
It follows that the signatures of all terms on the right-hand side of 
\eqref{C_r_k_def} are the same.    Therefore obviously $C_r(k)\neq 0$, and its signature
coincides with $(-1)^{r+[r/k]}$.    The theorem is now proved.
\end{proof}

Moreover, we can determine the explicit values of the constants $C_r(k)$.

\begin{theorem}\label{Th-coeff}
We have
\begin{align}\label{Th_formula2}
C_r(r)=(-1)^{r-1}\frac{1}{r} \qquad (r\geq 1),
\end{align}
\begin{align}\label{Th_formula3}
C_r(k)=\frac{(-1)^{k-1}}{k}\zeta_{r-k}\left(\frac{1}{k}\right) \qquad 
(r/2< k\leq r-1),
\end{align}
\begin{align}\label{Th_formula4}
C_r(1)=\frac{1}{r!} \qquad (r\geq 1),
\end{align}
and, for any $k\geq 2$, 
\begin{align}\label{Th_formula5}
&C_{kr}(k)=\frac{(-1)^{(k-1)r}}{k^r\cdot r!} \qquad (r\geq 1),\\
&C_{kr+\ell}(k)=\frac{(-1)^{(k-1)r}}{k^r\cdot r!}\zeta_{\ell}\left(\frac{1}{k}\right)
\qquad (r\geq 1, 1\leq \ell\leq k-1). \label{Th_formula6}
\end{align}
\end{theorem}

\begin{remark}
For any fixed $k\geq 2$, formulas \eqref{Th_formula5} and \eqref{Th_formula6} show
that $C_r(k)$, as a function in $r$, has a kind of ``periodicity'' mod $k$.
\end{remark}

\begin{remark}
Formulas \eqref{Th_formula4}, \eqref{Th_formula5} and \eqref{Th_formula6} exhausts
all the cases of $C_r(k)$, $k,r\in\mathbb{N}$, $1\leq k\leq r$.     However we still
prefer to include \eqref{Th_formula2} and \eqref{Th_formula3} in the statement of
Theorem \ref{Th-coeff}, because they themselves are elegant formulas, and they are
necessary in the proof of \eqref{Th_formula5} and \eqref{Th_formula6}.
\end{remark}

\begin{proof}[Proof of Theorem \ref{Th-coeff}]
The first formula \eqref{Th_formula2} was already shown in the proof of Theorem
\ref{Th-asymp}. To prove the remaining formulas, we use \eqref{C_r_k_def}.

Let $r/2< k\leq r-1$.   Then $D_{r-k}(k)=\zeta_{r-k}(1/k)$, hence 
\eqref{C_r_k_def} is
\begin{align}\label{C_r_k_def2}
C_r(k)
=\frac{1}{r}\left\{{\sum_j}^* (-1)^{j-1}\zeta\left(\frac{j}{k}\right)C_{r-j}(k)
+(-1)^{k-1}\zeta_{r-k}\left(\frac{1}{k}\right)\right\},
\end{align}
where the summation $\sum^*$ runs over $1\leq j\leq r-k$.
First consider the case $k=r-1$.   Then
\begin{align}
C_r(r-1)=\frac{1}{r}\left\{\zeta\left(\frac{1}{r-1}\right)C_{r-1}(r-1)
+(-1)^{r-2}\zeta\left(\frac{1}{r-1}\right)\right\}.
\end{align}
Since $C_{r-1}(r-1)=(-1)^{r-2} (r-1)^{-1}$ by \eqref{Th_formula2}, the above is equal to
$$
\frac{(-1)^{r-2}}{r}\cdot \left(\frac{1}{r-1}+1\right)\zeta\left(\frac{1}{r-1}\right)
=\frac{(-1)^{r-2}}{r-1}\zeta\left(\frac{1}{r-1}\right),
$$
which is \eqref{Th_formula3} for $k=r-1$.

Now we show \eqref{Th_formula3} by induction on $r-k$. 
We may apply induction assumption to $C_{r-j}(k)$ on the right-hand side of
\eqref{C_r_k_def2} for $1\leq j\leq r-k-1$, because $(r-j)/2< k\leq (r-j)-1$ for
those $j$ and $(r-j)-k < r-k$.
Hence
we apply \eqref{Th_formula3} (for $1\leq j\leq r-k-1$) and \eqref{Th_formula2}
(for $j=r-k$) to the right-hand side of \eqref{C_r_k_def2} to get
\begin{align}
C_r(k)&=\frac{1}{r}\left\{\sum_{j=1}^{r-k-1}(-1)^{j-1}\zeta\left(\frac{j}{k}\right)
\frac{(-1)^{k+1}}{k}\zeta_{r-k-j}\left(\frac{1}{k}\right)
+(-1)^{r-k-1}\zeta\left(\frac{r-k}{k}\right)(-1)^{k-1}\frac{1}{k}\right\}\\
&\;+\frac{(-1)^{k-1}}{r}\zeta_{r-k}\left(\frac{1}{k}\right)\notag\\
&=\frac{(-1)^{k+1}}{rk}\sum_{j=1}^{r-k}(-1)^{j-1}
\zeta_{r-k-j}\left(\frac{1}{k}\right)\zeta\left(\frac{j}{k}\right)
+\frac{(-1)^{k-1}}{r}\zeta_{r-k}\left(\frac{1}{k}\right).\notag
\end{align}
Since the sum on the right-hand side is equal to $(r-k) \zeta_{r-k}(1/k)$ by \eqref{zeta_id}, 
we obtain
$$
C_r(k)=(-1)^{k+1}\left(\frac{r-k}{rk}+\frac{1}{r}\right)
\zeta_{r-k}\left(\frac{1}{k}\right)
=\frac{(-1)^{k+1}}{k}\zeta_{r-k}\left(\frac{1}{k}\right),
$$
which is \eqref{Th_formula3}.

Next consider $C_r(1)$.    The case $r=1$ is included in \eqref{Th_formula2}.
Assume $r\geq 2$.    Then \eqref{C_r_k_def} implies $C_r(1)=r^{-1}C_{r-1}(1)$, from
which \eqref{Th_formula4} immediately follows.

Lastly we prove \eqref{Th_formula5} and \eqref{Th_formula6}.
First we notice that the case $r=1$ of \eqref{Th_formula5} and \eqref{Th_formula6}
is included in \eqref{Th_formula2} and \eqref{Th_formula3}, respectively.
Therefore now assume $r\geq 2$, and prove the formulas by induction on $r$.

From \eqref{C_r_k_def} we have
\begin{align}\label{C_r_k_def3}
C_{kr+\ell}(k)= \frac{1}{kr+\ell}\left\{{\sum_j}^* (-1)^{j-1}\zeta\left(\frac{j}{k}\right)
C_{kr+\ell-j}(k)
+(-1)^{k-1}C_{kr+\ell-k}(k)\right\},
\end{align}
because $k\leq (kr+\ell)/2$.    The summation $\sum^*$ runs over all $1\leq j\leq \ell$
(because $[(kr+\ell-j)/k]=[(kr+\ell)/k]=r$ should be satisfied).
In particular, when $\ell=0$ this sum is empty, hence
\begin{align}
C_{kr}(k)=\frac{(-1)^{k-1}}{kr}C_{k(r-1)}(k).
\end{align}
Therefore recursively we obtain
$$
C_{kr}(k)=\frac{(-1)^{k-1}}{kr}\cdot \frac{(-1)^{k-1}}{k(r-1)}\cdots
\frac{(-1)^{k-1}}{2k}C_k(k)
=\frac{(-1)^{(k-1)(r-1)}}{k^{r-1}r!}\cdot \frac{(-1)^{k-1}}{k}
=\frac{(-1)^{(k-1)r}}{k^r\cdot r!},
$$
which is \eqref{Th_formula5}.

Assume $\ell\geq 1$.    Here we also adopt the induction on $\ell$.   Since 
$\ell-j<\ell$, we use the induction (on $\ell$) assumption to $C_{kr+\ell-j}(k)$ and
the induction (on $r$) assumption to $C_{kr+\ell-k}(k)$ on the right-hand side of
\eqref{C_r_k_def3}.    Then
\begin{align}
C_{kr+\ell}(k)=\frac{1}{kr+\ell}\left\{\sum_{j=1}^{\ell}(-1)^{j-1}\zeta\left(\frac{j}{k}\right)
\frac{(-1)^{(k-1)r}}{k^r\cdot r!}\zeta_{\ell-j}\left(\frac{1}{k}\right)
+(-1)^{k-1}
\frac{(-1)^{(k-1)(r-1)}}{k^{r-1}(r-1)!}\zeta_{\ell}\left(\frac{1}{k}\right)\right\}.
\end{align}
The sum on the right-hand side is
$$
\frac{(-1)^{(k-1)r}}{k^r\cdot r!}\cdot \ell\zeta_{\ell}\left(\frac{1}{k}\right)
$$
by \eqref{zeta_id}, hence
\begin{align}
C_{kr+\ell}(k)=\frac{1}{kr+\ell}\cdot \frac{(-1)^{(k-1)r}}{k^{r-1}(r-1)!}
\left(\frac{\ell}{kr}+1\right)\zeta_{\ell}\left(\frac{1}{k}\right)
=\frac{(-1)^{(k-1)r}}{k^r\cdot r!}\zeta_{\ell}\left(\frac{1}{k}\right),
\end{align}
which is \eqref{Th_formula6}.
\end{proof}

Now we consider the distribution of real zeros of $\zeta_r(s)$.   By Theorem
\ref{Th<1/r} we know that $\zeta_r(s)\neq 0$ for $0\leq s<1/r$.    However,
the graphs for $\zeta_r(s)$ for $2\leq r\leq 6$ show the existence of zeros
in other intervals.    Denote by $I_r(k)$ the number of IAZs of $\zeta_r(s)$
in the interval $(1/k,1/(k-1))$.    From the graphs we may observe that all zeros
in the graphs seem simple, and
\begin{align*}
& I_2(2)=1,\\
& I_3(3)=1,I_3(2)=1,\\
& I_4(4)=1, I_4(3)=1, I_4(2)=2,\\
& I_5(5)=1, I_5(4)=1, I_5(3)=1, I_5(2)=2,\\
& I_6(6)=1, I_6(5)=1, I_6(4)=1, I_6(3)=2, I_6(2)=3.
\end{align*}
In the next subsection we will present the graphs of $\zeta_r(s)$, $7\leq r\leq 10$.
From those graphs we further observe:
\begin{align*}
& I_7(7)=1, I_7(6)=1, I_7(5)=1, I_7(4)=1, I_7(3)=2, I_7(2)=3,\\
& I_8(8)=1, I_8(7)=1, I_8(6)=1, I_8(5)=1, I_8(4)=2, I_8(3)=2, I_8(2)=4,\\
& I_9(9)=1, I_9(8)=1, I_9(7)=1, I_9(6)=1, I_9(5)=1, I_9(4)=2, I_9(3)=3, I_9(2)=4,\\
& I_{10}(10)=1, I_{10}(9)=1, I_{10}(8)=1, I_{10}(7)=1, I_{10}(6)=1, I_{10}(5)=2, 
I_{10}(4)=2, I_{10}(3)=3, I_{10}(2)=5.
\end{align*}

 Based on these data, here we propose the following conjecture.
 
\begin{conjecture}\label{conj-IAZ}
For any $r\geq 2$, all IAZs of $\zeta_r(s)$ are simple, and $I_r(k)=[r/k]$
$(2\leq k\leq r)$.
\end{conjecture}
 
The above data give a strong evidence for this conjecture, but so far we have
not found any rigorous proof of the conjecture. 
 
Let $IAZ(r)$ be the total number of IAZs, that is, the number of all real zeros
of $\zeta_r(s)$ in the interval $(0,1)$.

\begin{theorem}
If Conjecture \ref{conj-IAZ} is true, then we have
\begin{align}
IAZ(r)=\sum_{k=2}^r \left[\frac{r}{k}\right]=r\log r-2(1-\gamma)r+O(r^{1/2}),
\end{align}
where $\gamma=0.577215\ldots$ is Euler's constant.
\end{theorem}

\begin{proof}
The first equality is a direct consequence of the conjecture.    The second
equality follows from
\begin{align}\label{F-d}
\sum_{k=2}^r \left[\frac{r}{k}\right]=
\sum_{k=1}^r \left[\frac{r}{k}\right]-r
=\sum_{k=1}^r \sum_{\substack{\ell\leq r \\ \ell\equiv 0 ({\rm mod}\; k)}}1-r
=\sum_{\ell\leq r}d(\ell)-r,
\end{align}
where $d(\ell)$ denotes the number of positive divisors of $\ell$, and the known result
$$
\sum_{\ell=1}^r d(\ell)=r\log r+(2\gamma-1)r+O(r^{1/2})
$$
(\cite[Theorem 3.3]{Apos76}).
\end{proof}

\begin{remark}
The above error term $O(r^{1/2})$ is not best-possible.    It is conjectured that
the estimate $O(r^{1/4+\varepsilon})$ would hold for any $\varepsilon>0$, 
and the best known result is $O(r^{131/416+\varepsilon})$ due to
Huxley \cite{Hux03}.
\end{remark}

\begin{remark}
Write $F(r)=\sum_{k=2}^r [r/k]$, and let
$$
\Delta_{IAZ}(r)=IAZ(r)-IAZ(r-1), \qquad \Delta_F(r)=F(r)-F(r-1).
$$
If Conjecture \ref{conj-IAZ} is true, then $\Delta_{IAZ}(r)=\Delta_F(r)$, so it is
interesting to observe the behavior of $\Delta_F(r)$.
From \eqref{F-d} we see that
$\Delta_F(r)=d(r)-1$.    In particular, 
$\Delta_F(p)=1$ for any prime number $p$.
Further, $\Delta_F(r)$ is dominantly odd; it is even if and only if $r$ is a
square (because if the decomposition of $r$ into prime factors is
$r=p_1^{a_1}p_2^{a_2}\cdots p_k^{a_k}$, then $d(r)=(a_1+1)(a_2+1)\cdots (a_r+1)$).
\end{remark}

\subsection{Further examples}

Here we present the graphs of $\zeta_r(s)$, $7\leq r\leq 10$, $s\in [0,1]$.
We find that the behavior of those graphs agrees with our Conjecture \ref{conj-IAZ}.
 
 The seven-fold zeta-function $\zeta_7(s)$ has seven asymptotes: $\Re s=1, 1/2, 1/3, 1/4, 1/5, 1/6$ and $1/7$, and has nine IAZs for \( s\in [0,1]\) (Fig \ref{Fig9}):

\begin{itemize}[noitemsep]
    \item One IAZ $\in (1/7,1/6)$: \(\zeta_7(0,152170...)\approx 0.\)
    
    \item One IAZ $\in (1/6,1/5)$: \(\zeta_7(0,178811...)\approx 0. \)
    \item One IAZs $\in (1/5,1/4)$: \(\zeta_7(0,217987...) \approx0 \).
    \item One IAZs $\in (1/4,1/3)$: \(\zeta_7(0,298653...) \approx0 \).
    \item Two IAZs $\in (1/3,1/2)$: \(\zeta_7(0,365596..)\approx 0 \) and \(\zeta_7(0,442820...)\approx 0 \).
    
    \item Three IAZs $\in (1/2,1)$: \(\zeta_7(0,605776...)\approx 0 \), \(\zeta_7(0,736271...)\approx 0 \) and
\(\zeta_7(0,868958...)\approx 0.\) 
\end{itemize} 

Also \(\zeta_7(s)\) has one minimum and two maximums for \( s\in [0,1]\):

\begin{itemize}[noitemsep]
    \item One maximum between vertical asymptotes $1/3$ and $1/2$: \(\zeta_7(0,412055...)\approx  +4,875899...\)
    \item One minimum \(\zeta_7(0,663498...)\approx  -1,927124...\) and one maximum \(\zeta_7(0,847083...)\approx  +21,72816...\)
between vertical asymptotes $1/2$ and 1.
\end{itemize}

\begin{figure}[H]
\centering
\includegraphics{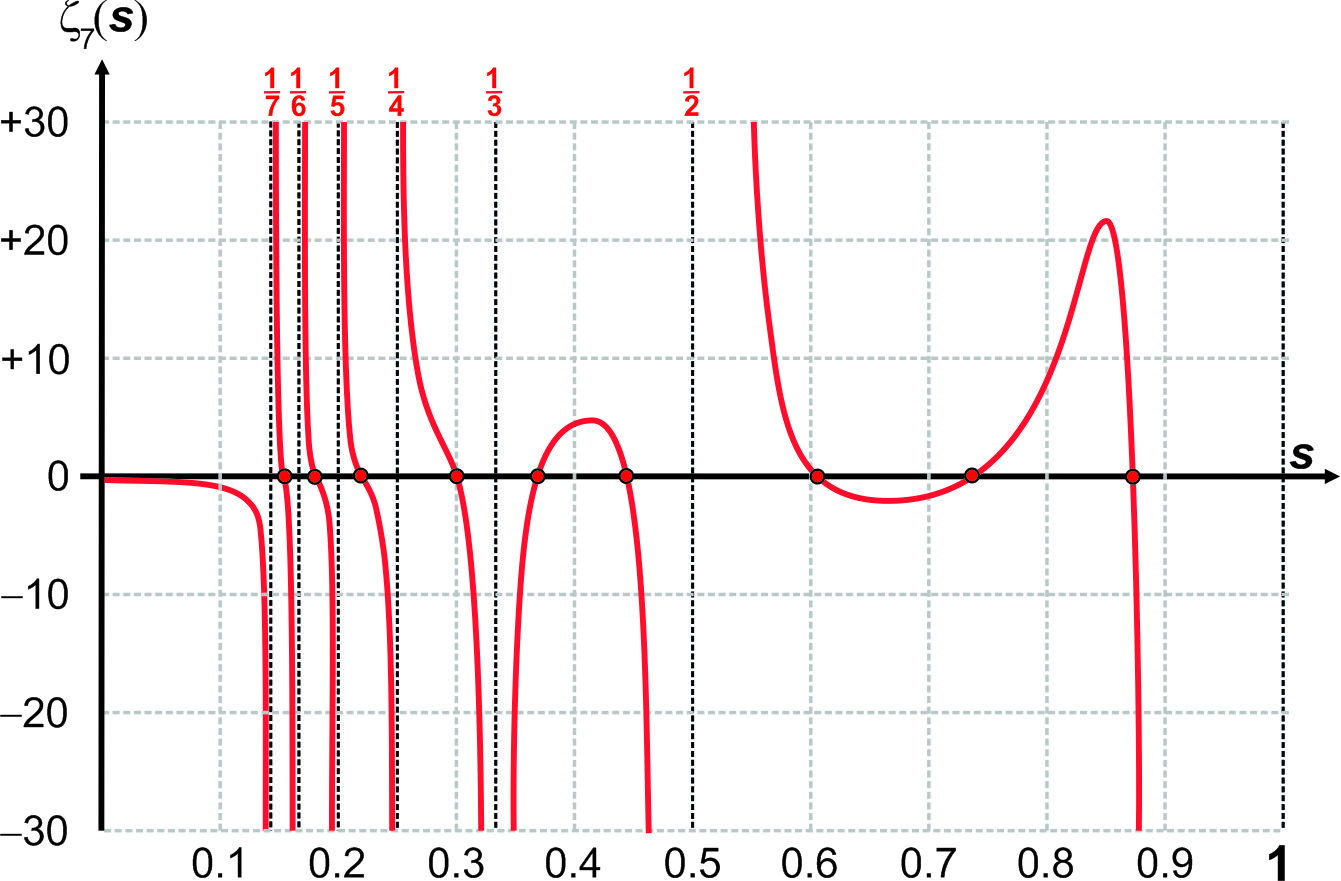}
\caption{The seven-fold zeta-function for $s\in [0,1]$}
\label{Fig9}
\end{figure}

The eight-fold zeta-function $\zeta_8(s)$ has eight asymptotes: $\Re s=1, 1/2, 1/3, 1/4, 1/5, 1/6, 1/7$ and $1/8$, and has twelve IAZs in \( s\in [0,1]\) (Fig \ref{Fig10}):

\begin{itemize}[noitemsep]
    \item One IAZ $\in (1/8,1/7)$: \(\zeta_8(0,132134...)\approx 0\)
    
    \item One IAZ $\in (1/7,1/6)$: \(\zeta_8(0,151738...)\approx 0.\)
    \item One IAZs $\in (1/6,1/5)$: \(\zeta_8(0,178822...) \approx0 \).
    \item One IAZs $\in (1/5,1/4)$: \(\zeta_8(0,218978...) \approx0 \).
    \item Two IAZs $\in (1/4,1/3)$: \(\zeta_8(0,266060...)\approx 0 \) and \(\zeta_8(0,295787..)\approx 0 \).
    \item Two IAZs $\in (1/3,1/2)$: \(\zeta_8(0,392752...)\approx 0 \) and \(\zeta_8(0,437321...)\approx 0 \).
    
    \item Four IAZs $\in (1/2,1)$: \(\zeta_8(0,538144...)\approx 0 \), \(\zeta_8(0,650658...)\approx 0 \) ,
\(\zeta_8(0,766794...)\approx 0 \) and \(\zeta_8(0,883665...)\approx 0 \).
\end{itemize} 

Also \(\zeta_8(s)\) has four minimums and one maximum for  \( s\in [0,1]\) :

\begin{itemize}[noitemsep]
    \item One minimun between vertical asymptotes $1/4$ and $1/3$: \(\zeta_8(0,277976...)\approx -2,253261...,\)
    \item One minimum between  vertical asymptotes $1/3$ and $1/2$: \(\zeta_8(0,420354...)\approx  -2,635752...,\) 
    
    \item Two minimus and one maximum between vertical asymptotes $1/2$ and 1: 
the first minimum at \( \zeta_8(0,551113...)\approx -12,188697...\), the maximum at \(\zeta_8(0,719417...)\approx  +1,334459...\), and the second minimum at \(\zeta_8(0,867892...)\approx  -45,821285....\)
\end{itemize}

\begin{figure}[H]
\centering
\includegraphics{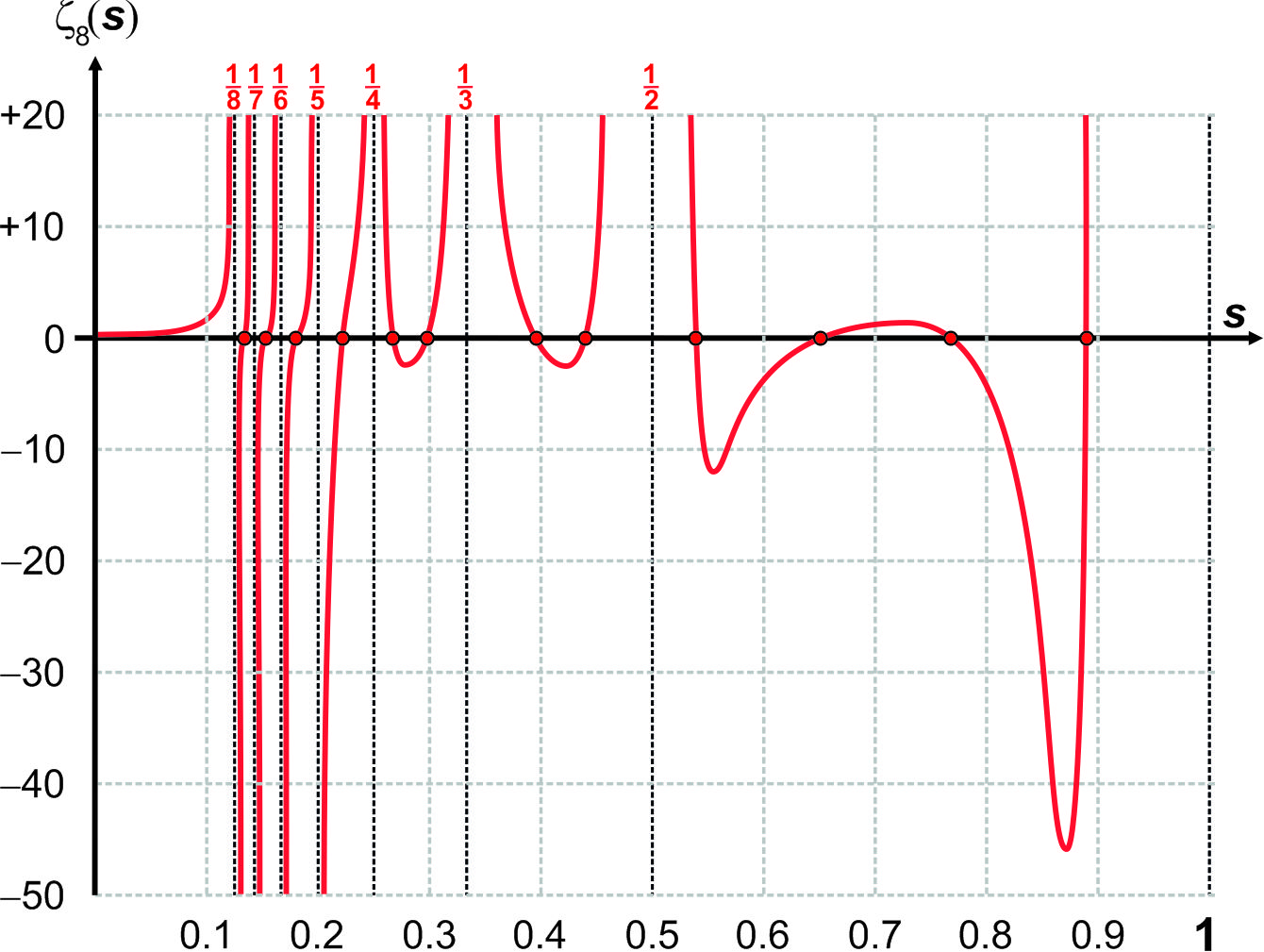}
\caption{The eight-fold zeta-function for $s\in [0,1]$}
\label{Fig10}
\end{figure}

Figure \ref{Fig11} shows the behavior of the nine-fold zeta-function, which first produces three
IAZs on the inter asymptotic interval \( s\in [1/3, 1/2]\).


\begin{figure}[H]
\centering
\includegraphics{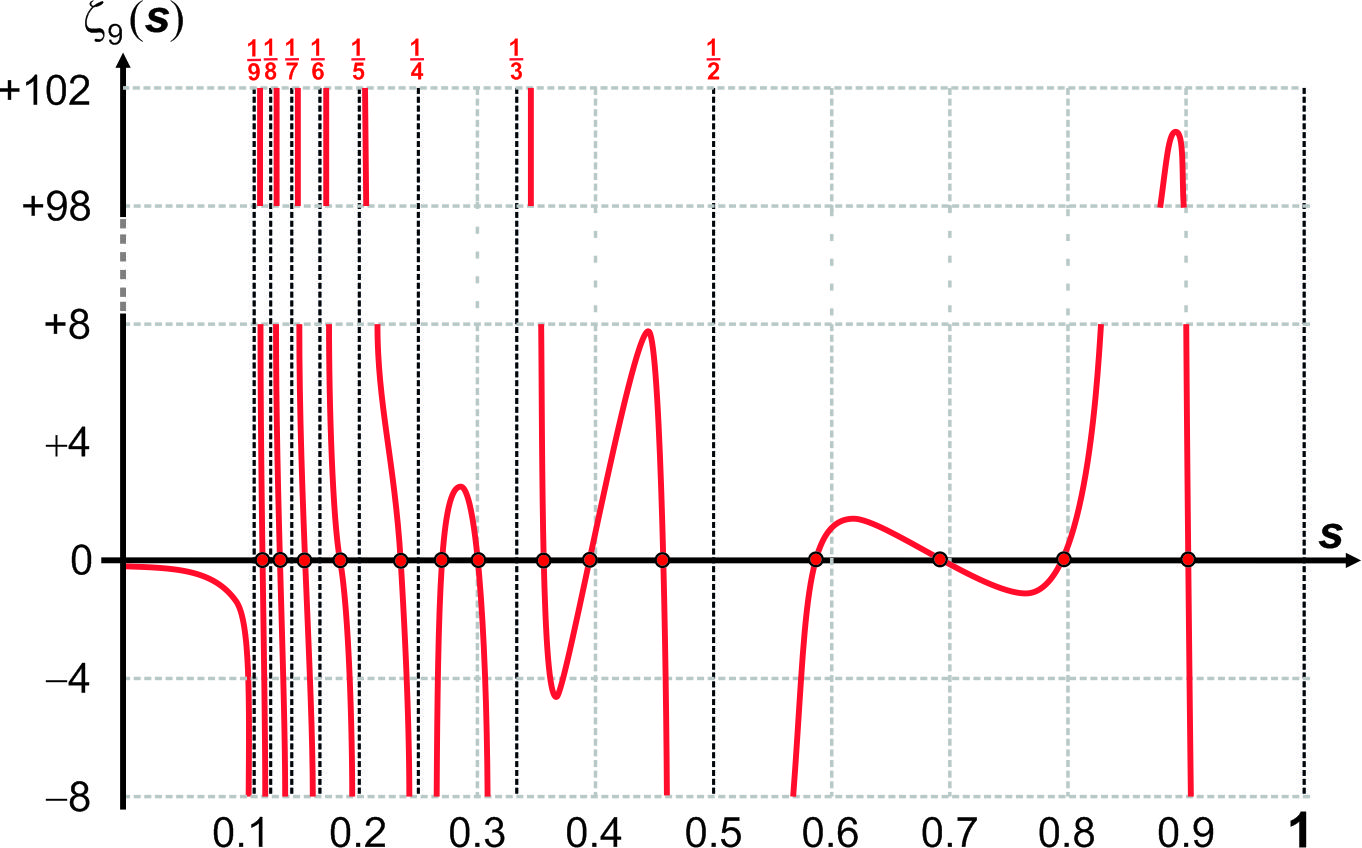}
\caption{The nine-fold zeta-function for $s\in [0,1]$}
\label{Fig11}
\end{figure}

Figure \ref{Fig12} shows the behavior of the ten-fold zeta-function, which first 
produces five IAZs at the inter-asymptotic interval \( s\in [1/2,1]\)  and
also, two IAZs at the inter asymptotic interval \( s\in [1/5,1/4]\).

%

\begin{figure}[H]
\centering
\includegraphics{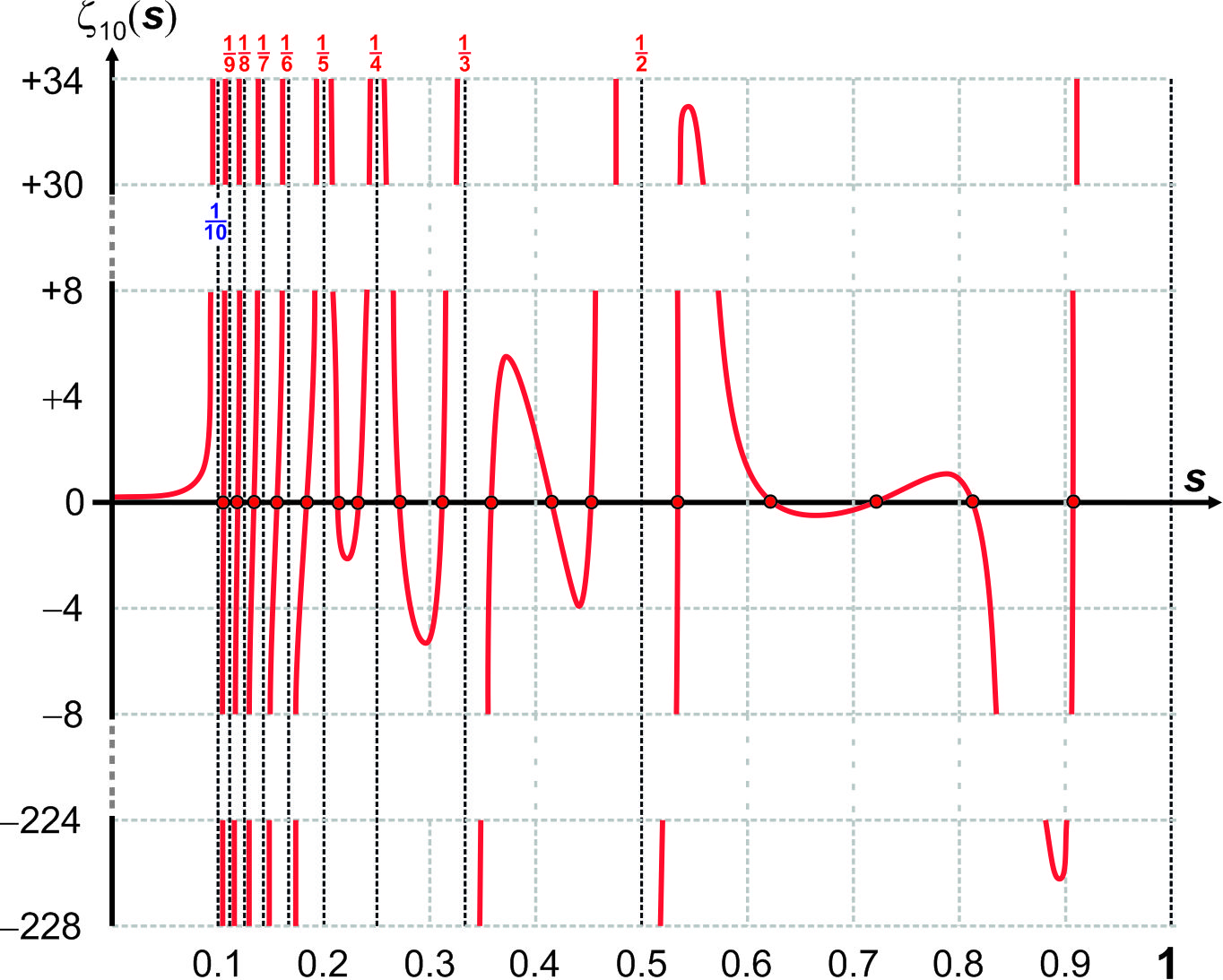}
\caption{The ten-fold zeta-function for $s\in [0,1]$}
\label{Fig12}
\end{figure}

\section{Multiple zeta-functions on the interval \(s \in (1,+\infty)\)}\label{sec-5}

In this section we discuss the behavior of $\zeta_r(s)$ for $s\in (1,+\infty)$.
From the definition it is obvious that $\zeta_r(s)$ is monotonically decreasing when $s$
increases in this interval, and
\begin{align*}
\lim_{s\to+\infty}\zeta_r(s)=\left\{
   \begin{array}{ll}
      1  &  (r=1)\\
      0  &  (r\geq 2).
  \end{array}\right.
  \end{align*}

%

We can refine this to the following asymptotic formula.

\begin{theorem}\label{Th-infty}
For any $r\geq 1$, 
we have $\zeta_r(s)\sim (r!)^{-s}$ as $s\to +\infty$.    In particular, \eqref{plus_inf}
holds.
\end{theorem}

\begin{proof}
Let $s>1$.    We first notice the inequality
\begin{align}\label{partialzeta}
\sum_{m=k}^{\infty}\frac{1}{m^s}=\frac{1}{k^s}+\sum_{m=k+1}^{\infty}\frac{1}{m^s}
\leq \frac{1}{k^s}+\int_{k}^{\infty}\frac{dt}{t^s}
= \frac{1}{k^s}\left(1+\frac{k}{s-1}\right).
\end{align}
Therefore
\begin{align*}
\zeta_r(s)&=\sum_{2\leq m_2<\cdots<m_r}\frac{1}{(m_2\cdots m_r)^s}
+\sum_{m_1\geq 2}\frac{1}{m_1^s}\sum_{(m_1<)m_2<\cdots< m_r}
\frac{1}{(m_2\cdots m_r)^s}\\
&\leq \left(1+\sum_{m_1\geq 2}\frac{1}{m_1^s}\right)
\sum_{2\leq m_2<\cdots<m_r}\frac{1}{(m_2\cdots m_r)^s}
\leq \left(1+\frac{1}{2^s}\left(1+\frac{2}{s-1}\right)\right)
\sum_{2\leq m_2<\cdots<m_r}\frac{1}{(m_2\cdots m_r)^s}.
\end{align*}
Using \eqref{partialzeta} again, we see that the last sum is
\begin{align*}
&=\frac{1}{2^s}\sum_{3\leq m_3<\cdots<m_r}\frac{1}{(m_3\cdots m_r)^s}+
\sum_{m_2\geq 3}\frac{1}{m_2^s}\sum_{(m_2<)m_3<\cdots< m_r}\frac{1}{(m_3\cdots m_r)^s}\\
&\leq \left(\frac{1}{2^s}+\frac{1}{3^s}\left(1+\frac{3}{s-1}\right)\right)
\sum_{3\leq m_3<\cdots<m_r}\frac{1}{(m_3\cdots m_r)^s}.
\end{align*}
Repeating this procedure, we arrive at
\begin{align}\label{ineq_inf}
\zeta_r(s)&\leq \left(1+\frac{1}{2^s}\left(1+\frac{2}{s-1}\right)\right)
\left(\frac{1}{2^s}+\frac{1}{3^s}\left(1+\frac{3}{s-1}\right)\right)\\
&\quad\times\cdots\times \left(\frac{1}{(r-1)^s}+\frac{1}{r^s}\left(1+\frac{r}{s-1}\right)\right)
\times\frac{1}{r^s}\left(1+\frac{r}{s-1}\right)\notag\\
&=\frac{1}{(r!)^s}\left(1+\frac{1}{2^s}\left(1+\frac{2}{s-1}\right)\right)
\left(1+\left(\frac{2}{3}\right)^s\left(1+\frac{3}{s-1}\right)\right)\notag\\
&\quad\times\cdots\times \left(1+\left(\frac{r-1}{r}\right)^s\left(1+\frac{r}{s-1}\right)\right)
\times\left(1+\frac{r}{s-1}\right).\notag
\end{align}
When $s\to +\infty$, 
$$
\left(\frac{k-1}{k}\right)^s \left(1+\frac{k}{s-1}\right)\to 0 \quad (2\leq k\leq r),
$$
hence $\zeta_r(s)\leq (r!)^{-s}(1+o(1))$ as $s\to +\infty$.
On the other hand it is obvious that $\zeta_r(s)\geq (r!)^{-s}$.
Therefore the assertion of the theorem follows.
\end{proof}

\begin{remark}
When $r=2$, Theorem \ref{Th-infty} is included in the proof of
\cite[Proposition 2.1]{MatSho14}.
\end{remark}

From \eqref{ineq_inf}, we can also prove:

\begin{theorem}\label{Th-r-infty}
For any fixed $s>1$, we have $\lim_{r\to\infty}\zeta_r(s)=0$.
\end{theorem}

\begin{proof}
Since $\zeta_r(s)$ is monotonically decreasing with respect to $s$, it is enough to show
the theorem for $s\in(1,2)$.
Then
$$
1+\left(\frac{k-1}{k}\right)^s\left(1+\frac{k}{s-1}\right)
\leq 1+\left(1+\frac{k}{s-1}\right)=2+\frac{k}{s-1}\leq \frac{2k}{s-1} \quad (2\leq k\leq r)
$$
and $1+r/(s-1)\leq 2r/(s-1)$.   Therefore from \eqref{ineq_inf} we have
$$
\zeta_r(s)\leq \frac{1}{(r!)^s}\cdot \left(\frac{2}{s-1}\right)^r r!\cdot r
= \left(\frac{2}{s-1}\right)^r \frac{r}{(r!)^{s-1}}.
$$
Here we apply the classical formula of Stirling
\begin{align}\label{Stirling}
r!\sim (2\pi r)^{1/2} r^r e^{-r} \qquad (r\to +\infty)
\end{align}
to get
\begin{align*}
\zeta_r(s)\lesssim \frac{r^{1-(s-1)/2}}{(2\pi)^{(s-1)/2}}\left(\frac{2e^{s-1}}{s-1}\right)^r
\frac{1}{r^{(s-1)r}}
\end{align*}
which, for any fixed $s>1$, tends to $0$ as $r\to\infty$.
\end{proof}

The assertions of Theorem \ref{Th-infty} and Theorem \ref{Th-r-infty} can be illustrated in Figure \ref{Fig13}, where we
symbolically write $\zeta_{\infty}(s)$ as the limit $r\to\infty$.

\begin{figure}[h]
\centering
\includegraphics{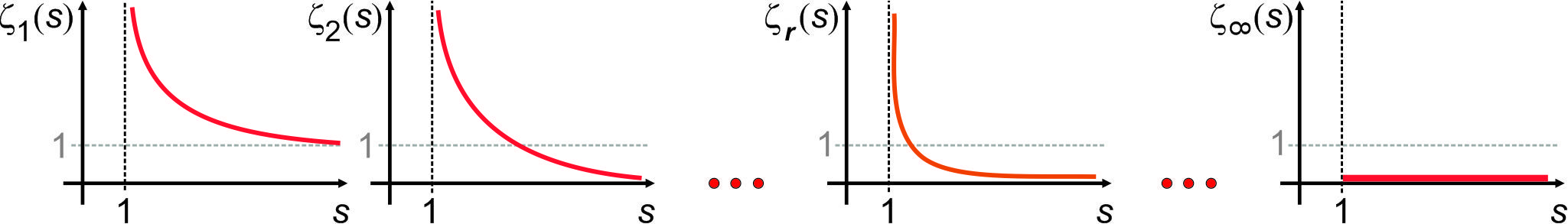}
\caption{Dynamics of the $r$-fold zeta-function \(\zeta_r(s)\) in the interval \(s \in (1,+\infty)\)   }
\label{Fig13}
\end{figure}

\section{Multiple zeta-functions on the interval \(s\in (-\infty,0]\)}\label{sec-6}

Now we consider the behavior of $\zeta_r(s)$ for $s\in(-\infty,0]$.
We already know (Theorem \ref{Th<1/r}) that the value $\zeta_r(0)$ is positive for
even $r$, and negative for odd $r$.
Here we list up the values of $\zeta_r(0)$ and $\zeta_r(-1)$ for $1\leq r\leq 14$ .

\begin{table}[hbt!]
\centering
\begin{tabular}{lcc|llc}
\multicolumn{1}{c}{$r$} & \(\zeta_r(0) \)             & \(\zeta_r(-1)\)        & \multicolumn{1}{c}{$r$} & \multicolumn{1}{c}{\(\zeta_r(0) \)  } & ~\(\zeta_r(-1)\)       \\ 
\hline
1                     &  $-$0,500000000000000 & $-$0,08333333333 & 8                     & +0,196380615234375        & $-$0,00005171790  \\
2                     & +0,375000000000000 & +0,00347222222 & 9                     & $-$0,185470581054687        & $-$0,00083949872  \\
3                     & $-$0,312500000000000 & +0,00268132716 & 10                    & +0,176197052001953        & +0,00007204895  \\
4                     & +0,273437500000000 & $-$0,00022947209 & 11                    & $-$0,168188095092773        & +0,00191443849  \\
5                     & $-$0,246093750000000 & $-$0,00078403922 & 12                    & +0,161180257797241        & $-$0,00016251626  \\
6                     & +0,225585937500000 & +0,00006972813 & 13                    & $-$0,154981017112732        & $-$0,00640336283  \\
7                     & $-$0,209472656250000 & +0,00059216643 & 14                    & +0,149445980787277        & +0,00054016476 
\end{tabular}
\end{table}

How about the zeros?    First of all, we note the following fact, which was
conjectured by \cite{AET01} and first proved by Kamano \cite{Kam06}.

\begin{prop}
For any $n\in\mathbb{N}$, we have $\zeta_r(-2n)=0$. 
\end{prop}

\begin{proof}
This is an immediate consequence of the fact that $\zeta(-2n)=0$ for any $n\in\mathbb{N}$.
In fact, since there is no pole of $\zeta_{r-j}(s)$ ($1\leq j\leq r$) on the negative
real line (Theorem \ref{Th-asymp} (i)), all terms on the right-hand side of
\eqref{zeta_id} are zero at $s=-2n$.
\end{proof}

Numerical computations show that there are more zeros on the negative real line, which we
call inter-trivial zeros (ITZs).
This phenomenon was first noticed by the first-named author and Sh{\=o}ji 
\cite{MatSho14} in the double zeta case.
We presented the graphs for the double and the triple cases in Section \ref{sec-3}.
We can compute further, and consequently, we may guess that the following conjecture is
plausible.

\begin{conjecture}\label{conj-ITZ}
There are $(r-1)$ ITZs of $\zeta_r(s)$ on the interval $(-2n,-2(n-1))$ for any 
$n\in\mathbb{N}$.   (Therefore, more and more ITZs appear when $r$ increases.)
\end{conjecture}


In fact, we can prove this conjecture with only one exception:

\begin{theorem}\label{Th-Toshiki}
There are exactly $(r-1)$ ITZs of $\zeta_r(s)$ on the interval $(-2n,-2(n-1))$ for any 
$n \geq 2$.
\end{theorem}

We will give the proof of this theorem in the next section.
The method of our proof of Theorem \ref{Th-Toshiki} cannot be applied to the case
$n=1$, and hence this case is still open.

We show the graphs of $r$-fold zeta-functions ($2\leq r\leq 10$) on the interval
$[-2,0]$ (see Figure \ref{Fig_neg}), and furthermore we present the data of ITZs
for the $r$-fold zeta-functions
\(\zeta_r(s)\), $2\leq r\leq 10$, on the same interval
(see Figure \ref{Fig4}).
This data clearly agree with the case $n=1$ of Conjecture \ref{conj-ITZ}.

\begin{figure}[h]
\centering
\includegraphics{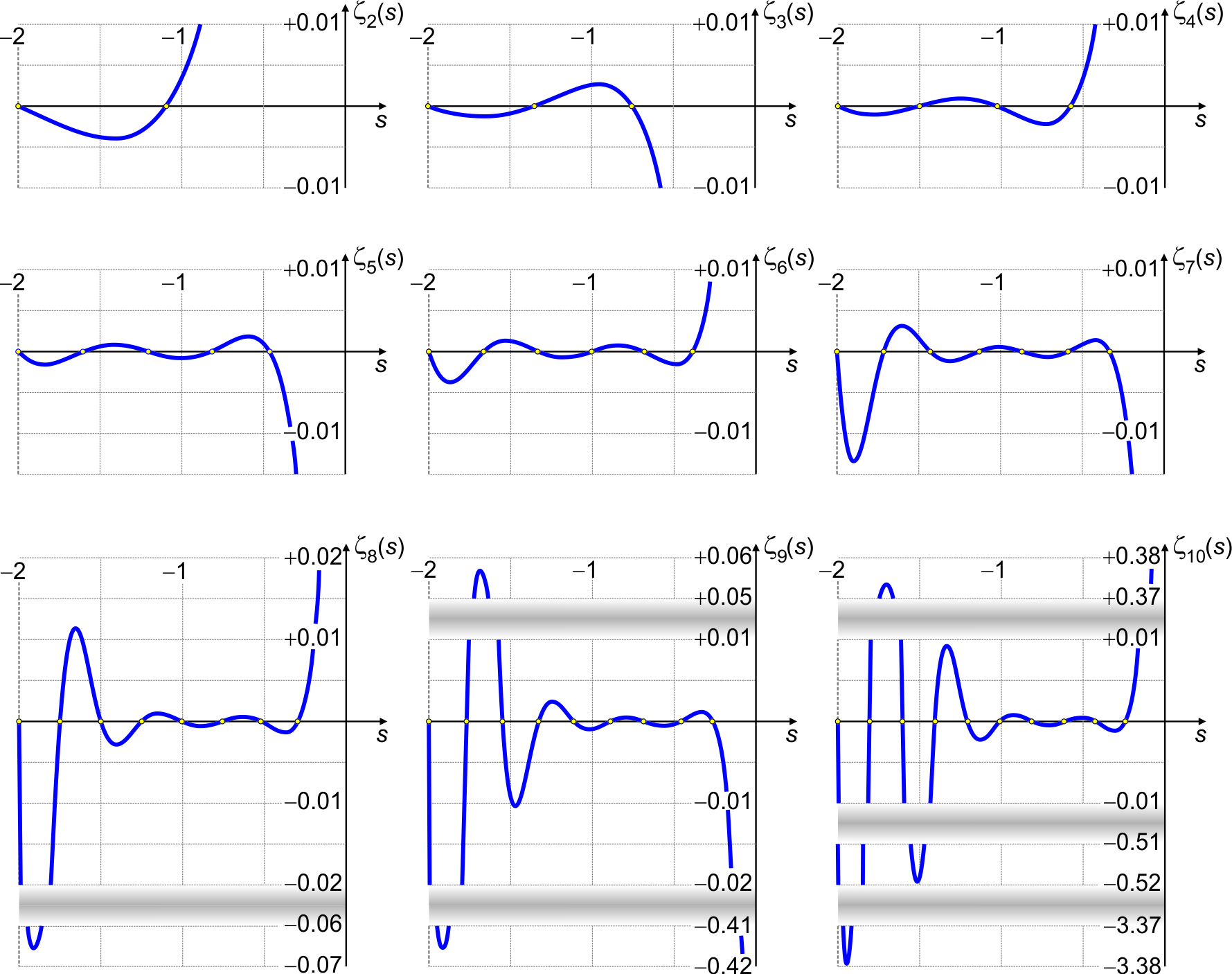}
\caption{The behavior of $r$-fold zeta-functions ($2\leq r\leq 10$) on the interval
$[-2,0]$ }
\label{Fig_neg}
\end{figure}

\begin{figure}[h]
\centering
\includegraphics{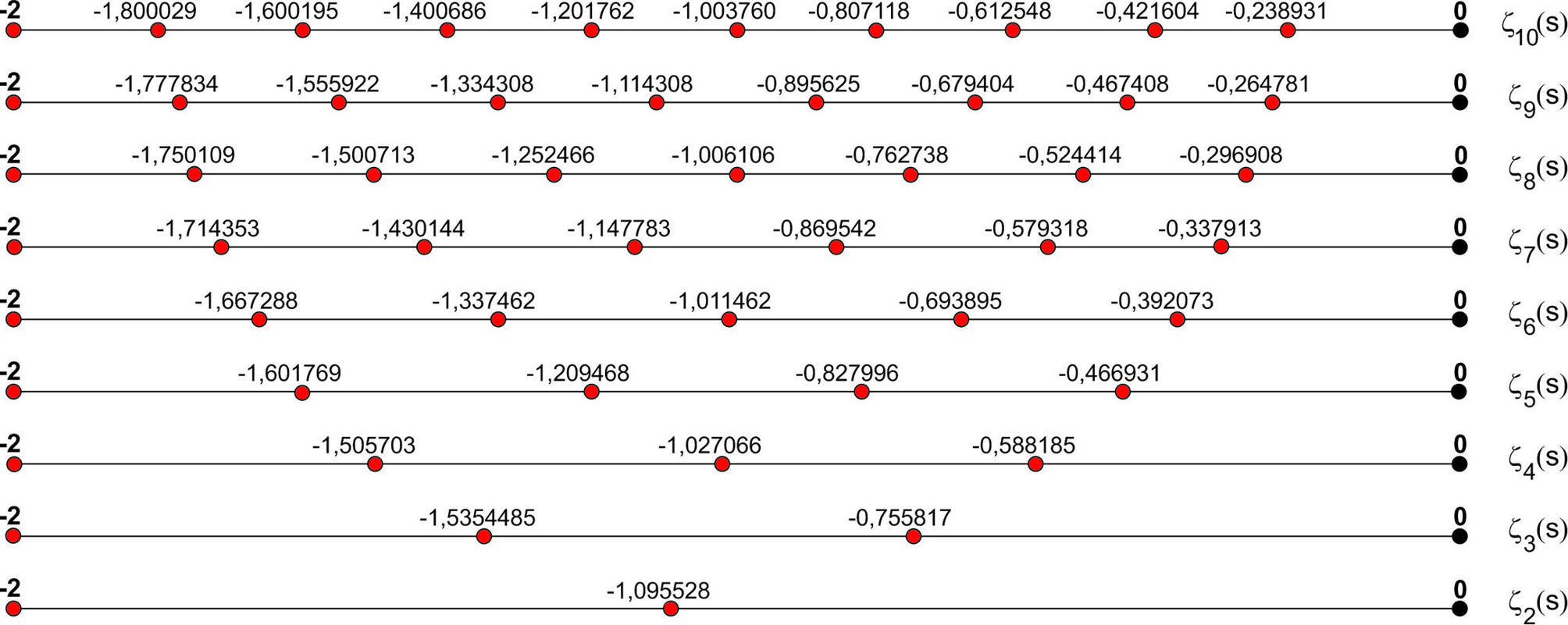}
\caption{Inter-trivial zeros of multiple zeta-functions $\zeta_r(s)$ ($2\leq r\leq 10$)
in the interval $[-2,0]$}
\label{Fig4}
\end{figure}

When $s\to -\infty$, the values of $\zeta_r(s)$ become highly oscillating, and take
larger and larger absolute values.
This situation can be well illustrated in the following theorem.

\begin{theorem}\label{Th-negativeodd}
Let $r\geq 1$.    When $k$ takes odd integer values and $k\to+\infty$, we have

{\rm (i)}
If $r$ is odd, 
\begin{align}\label{Th-oddsymp}
\zeta_r(-k)\sim \frac{1}{r}\zeta(-rk)
\sim (-1)^{(rk+1)/2}\cdot\frac{2k^{1/2}}{\sqrt{2\pi}}\cdot r^{rk-1/2}
\left(\frac{k}{2\pi e}\right)^{rk}.
\end{align}

{\rm (ii)}
\begin{align}\label{Th-2symp}
\zeta_2(-k)=\frac{1}{2}\zeta(-k)^2
\sim \frac{k}{\pi}\left(\frac{k}{2\pi e}\right)^{2k}.
\end{align}

{\rm (iii)}
If $r\geq 4$ is even,
\begin{align}\label{Th-evensymp}
\zeta_r(-k)\sim \frac{1}{r-1}\zeta(-k)\zeta(-(r-1)k)
\sim (-1)^{(r/2)-1}\cdot\frac{2k}{\pi}\cdot(r-1)^{(r-1)k-1/2}
\left(\frac{k}{2\pi e}\right)^{rk}.
\end{align}
\end{theorem}

To prove this theorem, we first prepare:

\begin{lemma}\label{Lemma-elem}
Let $r\geq 6$ and $3\leq j\leq r-3$.   Then 
$(r-j)^{r-j} j^j < (r-1)^{r-1}$.
\end{lemma}

\begin{proof}
The claim is equivalent to the inequality
\begin{align}\label{log-style}
(r-j)\log(r-j)+j\log j < (r-1)\log (r-1)
\end{align}
for $3\leq j\leq r-j$.
Since the function $x\log x$ is convex downward for $x>0$, it is sufficient to prove
\eqref{log-style} for $j=3$:
\begin{align}\label{log3-style}
(r-3)\log(r-3)+3\log 3 < (r-1)\log (r-1).
\end{align}
When $r=6$, this is $6\log 3 < 5\log 5$, which is true.    When $r\geq 7$, we have
$3\log 3 < 2\log(r-1)$, and so
$$
(r-3)\log (r-3)+3\log 3 < (r-3)\log(r-1)+2\log (r-1)=(r-1)\log (r-1).
$$
\end{proof}

\begin{proof}[Proof of Theorem \ref{Th-negativeodd}]
First consider the case $r=1$.    Then it is well-known that
$$
\zeta_1(-k)=\zeta(-k)=-\frac{B_{k+1}}{k+1}
$$
for odd $k$,
where $B_{k+1}$ is the $(k+1)$-th Bernoulli number (see \cite[Theorem 12.16]{Apos76}).   
Since the asymptotic formula
$B_{2k}\sim (-1)^{k+1}2(2k)!/(2\pi)^{2k}$, as $k\to\infty$, is 
known (see \cite[Theorem 12.18]{Apos76}),
combining with \eqref{Stirling} we obtain
$$
B_{k+1}\sim (-1)^{(k-1)/2}\frac{2k^{3/2}k^k}{(2\pi)^{1/2}(2\pi e)^k},
$$
and hence
\begin{align}\label{Th-1symp}
\zeta(-k)\sim (-1)^{(k+1)/2}\frac{2k^{1/2}k^k}{(2\pi)^{1/2}(2\pi e)^k}
= (-1)^{(k+1)/2}\frac{2k^{1/2}}{\sqrt{2\pi}}\left(\frac{k}{2\pi e}\right)^k,
\end{align}
which is the case $r=1$ of \eqref{Th-oddsymp}.

Then, from \eqref{2_id} we have
$$
\zeta_2(-k)=\frac{1}{2}\left(\zeta(-k)^2-\zeta(-2k)\right)=\frac{1}{2}\zeta(-k)^2
\sim\frac{k}{\pi}\left(\frac{k}{2\pi e}\right)^{2k},
$$
which is \eqref{Th-2symp}.    Next, using \eqref{3_id} and applying \eqref{Th-1symp}
and \eqref{Th-2symp}, we obtain
\begin{align*}
\zeta_3(-k)&=\frac{1}{3}\left(\zeta_2(-k)\zeta(-k)+\zeta(-3k)\right)\\
&\sim \frac{1}{3}\left\{\frac{k}{\pi}\left(\frac{k}{2\pi e}\right)^{2k}
(-1)^{(k+1)/2}\frac{2k^{1/2}}{\sqrt{2\pi}}\left(\frac{k}{2\pi e}\right)^k
+(-1)^{(3k+1)/2}\frac{2(3k)^{1/2}}{\sqrt{2\pi}}\left(\frac{3k}{2\pi e}\right)^{3k}
\right\}\\
&=\frac{1}{3}(-1)^{(k+1)/2}\frac{2}{\sqrt{2\pi}}\left(\frac{k}{2\pi e}\right)^{3k}
\left\{\frac{k^{3/2}}{\pi}-k^{1/2}3^{3k+1/2}\right\}\\
&\sim (-1)^{(3k+1)/2}\frac{2k^{1/2}}{\sqrt{2\pi}}3^{3k-1/2}
\left(\frac{k}{2\pi e}\right)^{3k},
\end{align*}
which is the case $r=3$ of \eqref{Th-oddsymp}.    This calculation also shows
$\zeta_3(-k)\sim (1/3)\zeta(-3k)$, and hence, from \eqref{4_id}, 
$$
\zeta_4(-k)=\frac{1}{4}\left(\zeta_3(-k)\zeta(-k)+\zeta(-k)\zeta(-3k)\right)
\sim \frac{1}{3}\zeta(-k)\zeta(-3k).
$$
Therefore using \eqref{Th-1symp} we obtain
$$
\zeta_4(-k)\sim -\frac{2k}{\pi}3^{3k-1/2}\left(\frac{k}{2\pi e}\right)^{4k},
$$
which is the case $r=4$ of \eqref{Th-evensymp}.

Now we prove the general case by induction on $r$.    The formula \eqref{zeta_id}
implies
\begin{align}\label{zeta_id_k}
\zeta_r(-k)=\frac{1}{r}\sum_{\stackrel{i\leq j\leq r}{j:{\rm odd}}}
(-1)^{j-1}\zeta_{r-j}(-k)\zeta(-jk).
\end{align}
When $r$ is odd, we separate the last term from the above sum to get
\begin{align}\label{zeta_id_k_odd}
\zeta_r(-k)=\frac{1}{r}\sum_{\stackrel{i\leq j\leq r-1}{j:{\rm odd}}}
(-1)^{j-1}\zeta_{r-j}(-k)\zeta(-jk)+\frac{1}{r}\zeta(-rk).
\end{align}
Note that $r-j$ is even in each term in the above sum.
Therefore, By the induction assumption, each term in the above sum is
\begin{align}\label{orderestimate1}
&\ll_r k(r-j-1)^{(r-j-1)k-1/2}\left(\frac{k}{2\pi e}\right)^{(r-j)k}\cdot (jk)^{1/2}
\left(\frac{jk}{2\pi e}\right)^{jk}\\
&=(r-j-1)^{(r-j-1)k-1/2}k^{3/2} j^{jk+1/2}\left(\frac{k}{2\pi e}\right)^{rk}.\notag
\end{align}
Since
$$
(r-j-1)^{(r-j-1)k-1/2} j^{jk+1/2}\leq r^{(r-j-1)k-1/2} r^{jk+1/2}=r^{(r-1)k},
$$
the order of the right-hand side of \eqref{orderestimate1} with respect to $k$ 
is less than the order of the term $(1/r)\zeta(-rk)$ (which is
$k^{1/2}(rk/2\pi e)^{rk}$).
Therefore from \eqref{zeta_id_k_odd} we have 
$\zeta_r(-k)\sim (1/r)\zeta(-rk)$, and hence \eqref{Th-oddsymp} follows by applying
\eqref{Th-1symp} to $\zeta(-rk)$.

We proceed to the case of even $r$.    In this case the term corresponding to $j=r$
does not appear on the right-hand side of \eqref{zeta_id_k}.    We divide the sum as
\begin{align*}
\zeta_r(-k)=\frac{1}{r}\left(\zeta_{r-1}(-k)\zeta(-k)+\zeta(-k)\zeta(-(r-1)k)\right)
+E_r,
\end{align*}
where $E_r$ consists of the terms corresponding to $3\leq j\leq r-3$, $j$:odd.
Then, since $r-1$ is odd, we use the induction assumption 
$\zeta_{r-1}(-k)\sim (r-1)^{-1}\zeta(-(r-1)k)$ to obtain
\begin{align}\label{zeta_id_k_even}
\zeta_r(-k)\sim \frac{1}{r-1}\zeta(-k)\zeta(-(r-1)k)+E_r.
\end{align}
If the order of $E_r$ is less than the order of $\zeta(-k)\zeta(-(r-1)k)$, then
$\zeta_r(-k)\sim (r-1)^{-1}\zeta(-k)\zeta(-(r-1)k)$, and hence \eqref{Th-evensymp}
follows by applying \eqref{Th-1symp}.
The order of $\zeta(-k)\zeta(-(r-1)k)$ is given by the right-hand side of 
\eqref{Th-evensymp}, while the order of each term $\zeta_{r-j}(-k)\zeta(-jk)$
($3\leq j\leq r-3$, $j$:odd) in the sum $E_r$ is
\begin{align*}
&\ll k^{1/2}(r-j)^{-1/2}\left(\frac{(r-j)k}{2\pi e}\right)^{(r-j)k}
(jk)^{1/2}\left(\frac{jk}{2\pi e}\right)^{jk}\\
&\ll_r k(r-j)^{(r-j)k} j^{jk}\left(\frac{k}{2\pi e}\right)^{rk}.
\end{align*}
Since the order of $(r-j)^{(r-j)k} j^{jk}$ with respect to $k$ is less than the order
of $(r-1)^{(r-1)k}$ by Lemma \ref{Lemma-elem}, we obtain the desired assertion.
\end{proof}

\section{Proof of Theorem \ref{Th-Toshiki}}\label{sec5}

In this last section we provide a proof of Theorem \ref{Th-Toshiki}. 
For positive integers $a, b, k > 0$, we let
\[
	R_k = R_k(a+b) := \left\{s \in \mathbb{C} \bigg| -2 - \frac{2k+1}{a+b} \leq \Re s \leq -2 -\frac{2k-1}{a+b}, -\frac{1}{a+b} \leq \Im s \leq \frac{1}{a+b} \right\},
\]
and
\[
	R_0 = R_0(a+b):= \left\{s \in \mathbb{C} \bigg| -2 - \frac{1}{a+b} \leq \Re s \leq -2, -\frac{1}{a+b} \leq \Im s \leq \frac{1}{a+b} \right\}.
\]
First, we prepare the following three lemmas.

\begin{lemma}\label{lemma1}
	Fix positive integers $a, b $. For any $k \geq 0$ and any $s \in \partial R_k$ on the boundary, we have
	\[
		\left| \frac{\sin \left(\frac{\pi}{2} (a+b) s \right)}{\sin \left(\frac{\pi}{2} as \right) \sin \left(\frac{\pi}{2} bs \right)} \right| > \frac{2}{a+b}.
	\]
\end{lemma}

\begin{proof} We divide the proof into three cases.

	1. For $s = x  \pm \frac{i}{a+b}$ $(x \leq -2)$, by the fact that $|\sin(x+iy)|^2 = \sin^2 x + \sinh^2 y$, we have
\begin{align}\label{proof_first_case}
		\left| \frac{\sin \left(\frac{\pi}{2} (a+b) s\right)}{\sin \left(\frac{\pi}{2} as \right) \sin \left(\frac{\pi}{2} bs \right)} \right|^2 &= \frac{\sin^2 \left(\frac{\pi}{2}(a+b)x \right) + \sinh^2 \left(\frac{\pi}{2} \right)}{\left(\sin^2 \left(\frac{\pi}{2} ax \right) + \sinh^2 \left(\frac{\pi}{2} \frac{a}{a+b} \right) \right) \left(\sin^2 \left(\frac{\pi}{2} bx \right) + \sinh^2 \left(\frac{\pi}{2} \frac{b}{a+b} \right) \right)}\\
			&\geq \frac{\sinh^2 \left(\frac{\pi}{2} \right)}{\cosh^2 \left(\frac{\pi}{2} \frac{a}{a+b}\right) \cosh^2 \left(\frac{\pi}{2} \frac{b}{a+b}\right)}.\notag
\end{align}
The	addition formula for $\sinh y$ gives
\begin{align*}
&\sinh\left(\frac{\pi}{2}\right)=
\sinh\left(\frac{\pi}{2}\frac{a}{a+b}+\frac{\pi}{2}\frac{b}{a+b}\right)\\
&=\sinh\left(\frac{\pi}{2}\frac{a}{a+b}\right)\cosh\left(\frac{\pi}{2}\frac{b}{a+b}\right)
+\cosh\left(\frac{\pi}{2}\frac{a}{a+b}\right)\sinh\left(\frac{\pi}{2}\frac{b}{a+b}\right).
\end{align*}
Therefore the right-hand side of \eqref{proof_first_case} is equal to
$$
\left(\tanh \left(\frac{\pi}{2} \frac{a}{a+b} \right) + \tanh \left(\frac{\pi}{2} \frac{b}{a+b} \right) \right)^2,
$$
which implies
	\[
		\left| \frac{\sin \left(\frac{\pi}{2} (a+b) s\right)}{\sin \left(\frac{\pi}{2} as \right) \sin \left(\frac{\pi}{2} bs \right)} \right| \geq \tanh \left(\frac{\pi}{2} \frac{a}{a+b} \right) + \tanh \left(\frac{\pi}{2} \frac{b}{a+b} \right) \geq 2 \tanh \left(\frac{\pi}{2} \frac{1}{a+b} \right).
	\]
	Since $\tanh (\pi/2y) > 1/y$ holds for $y \geq 2$, we now obtain
	\[
		\left| \frac{\sin \left(\frac{\pi}{2} (a+b) s\right)}{\sin \left(\frac{\pi}{2} as \right) \sin \left(\frac{\pi}{2} bs \right)} \right| > \frac{2}{a+b}.
	\]
	
	2. For $s = -2 - \frac{2k-1}{a+b} + \frac{it}{a+b}$ $(k > 0, -1 \leq t \leq 1)$, in a similar manner, we get
	\begin{align*}
		\left| \frac{\sin \left(\frac{\pi}{2} (a+b) s\right)}{\sin \left(\frac{\pi}{2} as \right) \sin \left(\frac{\pi}{2} bs \right)} \right|^2 &= \frac{1 + \sinh^2 \left(\frac{\pi}{2} t \right)}{\left(\sin^2 \left(\frac{\pi(2k-1)}{2} \frac{a}{a+b} \right) + \sinh^2 \left(\frac{\pi}{2} \frac{a}{a+b} t \right) \right) \left(\sin^2 \left(\frac{\pi(2k-1)}{2} \frac{b}{a+b} \right) + \sinh^2 \left(\frac{\pi}{2} \frac{b}{a+b} t \right) \right)}\\
			&> \frac{\cosh^2 \left(\frac{\pi}{2} t \right)}{\cosh^2 \left(\frac{\pi}{2} \frac{a}{a+b} t \right) \cosh^2 \left(\frac{\pi}{2} \frac{b}{a+b} t \right)} = \frac{4\cosh^2 \left(\frac{\pi}{2} t \right)}{\left(\cosh \left(\frac{\pi}{2} t\right) + \cosh \left(\frac{\pi}{2} \frac{a-b}{a+b} t \right) \right)^2} \geq 1 \geq \frac{4}{(a+b)^2}.
	\end{align*}
	Here, the first inequality is valid because the equation 
$$\sin^2 \left(\frac{\pi (2k-1)}{2} \frac{a}{a+b} \right) = \sin^2 \left(\frac{\pi (2k-1)}{2} \frac{b}{a+b} \right) = 1$$ 
does not happen.
	
	3. For $s = -2 + \frac{it}{a+b}$ $(-1 \leq t \leq 1)$, 
	again using the addition formula, we have
	\[
		\left| \frac{\sin \left(\frac{\pi}{2} (a+b) s\right)}{\sin \left(\frac{\pi}{2} as \right) \sin \left(\frac{\pi}{2} bs \right)} \right| = \left| \frac{\sinh \left(\frac{\pi}{2} t \right)}{\sinh \left( \frac{\pi}{2} \frac{a}{a+b} t \right)\sinh \left( \frac{\pi}{2} \frac{b}{a+b} t \right)} \right| = \left| \frac{1}{\tanh \left(\frac{\pi}{2} \frac{a}{a+b} t \right)} + \frac{1}{\tanh \left(\frac{\pi}{2} \frac{b}{a+b} t \right)} \right| > 2 \geq \frac{2}{a+b}.
	\]
Note that this argument is valid even when $t=0$; in this case we understand that the
left-hand side is $+\infty$.	
	
The proof of Lemma \ref{lemma1} is thus complete.
\end{proof}

We now replace $R_0$ with a smaller rectangle
\[
	R_0^{(\varepsilon)} = \left\{ s \in \mathbb{C}  \bigg| -2 - \frac{1}{a+b} \leq \Re s \leq -2-\varepsilon, -\frac{1}{a+b} \leq \Im s \leq \frac{1}{a+b} \right\}
\]
for a small enough $\varepsilon > 0$ such that Lemma \ref{lemma1} also holds for $s \in \partial R_0^{(\varepsilon)}$. For simplicity, we refer to $R_0^{(\varepsilon)}$ as $R_0$.

\begin{lemma}\label{lemma2}
	Fix positive integers $a, b $. For any $s \in \bigcup_{k \geq 0} R_k$, we have
	\[
		\left|\frac{\Gamma(1-(a+b)s)}{\Gamma(1-as) \Gamma(1-bs)}\right| \left|\frac{\zeta(1-(a+b)s)}{\zeta(1-as) \zeta(1-bs)}\right| > \frac{1}{2\pi} (a+b)^2.
	\]
\end{lemma}

\begin{proof}
	First, the zeta-factor is evaluated as follows. For $x \geq 3$ and $-1 \leq y \leq 1$, 
	\[
		\left| |\zeta(x+iy)| - 1 \right| \leq \left| \sum_{n=2}^\infty n^{-x-iy} \right| \leq \sum_{n=2}^\infty n^{-x} \leq \zeta(3) - 1,
	\]
hence $2-\zeta(3)\leq |\zeta(x+iy)| \leq \zeta(3)$.
Therefore we get
	\[
		\left|\frac{\zeta(1-(a+b)s)}{\zeta(1-as) \zeta(1-bs)}\right| \geq \frac{2-\zeta(3)}{\zeta(3)^2}.
	\]
	
	Secondly, since it is known that
	\[
		|\Gamma(x+iy)| = |\Gamma(x)| \prod_{j=0}^\infty \left(1+\frac{y^2}{(x+j)^2} \right)^{-1/2},
	\]
	for $s = -x + i \frac{t}{a+b}$ $(x \geq 2, -1 \leq t \leq 1)$ we have
	\[
		\left|\frac{\Gamma(1-(a+b)s)}{\Gamma(1-as) \Gamma(1-bs)}\right| \geq G(x)H(x),
	\]
where
$$
G(x)=\frac{\Gamma((a+b)x+1)}{\Gamma(ax+1) \Gamma(bx+1)}, \quad
H(x)=\prod_{j=0}^\infty \left(1 + \frac{1}{((a+b)x+1+j)^2} \right)^{-1/2}.
$$

Write $z := (a+b)x + 1$.   Then $z \geq 5$, and so
\begin{align}\label{proof_5_2}
		\sum_{j=0}^\infty \log \left(1 + \frac{1}{(z+j)^2} \right) < \sum_{j=0}^\infty \frac{1}{(z+j)^2} \leq \sum_{j=0}^\infty \frac{1}{(5+j)^2}=: A, 	
\end{align}
say.   Note that
$$
A= \zeta(2)-\sum_{l=1}^4\frac{1}{l^2}=\frac{\pi^2}{6} - \frac{205}{144}.
$$
From \eqref{proof_5_2} we obtain $H(x)>e^{-A/2}$.

	We next claim that the $G(x)$ is a monotonically increasing function in $x \geq 2$. This follows from the fact that its logarithmic derivative is always non-negative, that is,
	\[
		(a+b) \psi((a+b)x+1) \geq a \psi(ax+1) + b \psi(bx+1),
	\]
where $\psi(x)$ is the digamma function.   Since $\psi(x+1)$ is convex and monotonically increasing, we get
	\[
		\frac{a \psi(ax+1) + b \psi(bx+1)}{a+b} \leq \psi \left(\frac{a^2 + b^2}{a+b} x + 1 \right) \leq \psi ((a+b)x+1)
	\]
(the convexity gives the first inequality, and the monotonicity gives the second
inequality).   Therefore the claim follows.

We evaluate $G(x)$ at $x=2$:
	\[
		G(2)= {2a+2b \choose 2a} \geq {2a+2b \choose 2} = \frac{(2a+2b)(2a+2b-1)}{2} \geq \frac{3}{2\pi} (a+b)^2.
	\]
	Thus we have
	\begin{align*}
		\left|\frac{\Gamma(1-(a+b)s)}{\Gamma(1-as) \Gamma(1-bs)}\right| \left|\frac{\zeta(1-(a+b)s)}{\zeta(1-as) \zeta(1-bs)}\right| &\geq \frac{3}{2\pi}(a+b)^2 \cdot e^{-A/2} \cdot \frac{2-\zeta(3)}{\zeta(3)^2}\\
			& = 1.483 \dots \times \frac{1}{2\pi} (a+b)^2 > \frac{1}{2\pi} (a+b)^2,
	\end{align*}
	which concludes the proof.
\end{proof}

\begin{lemma}\label{lemma3}
	Let $r \geq 2$ be a fixed integer. For any $0 < j < r, k \geq 0$ and any $s \in \partial R_k(r)$ on the boundary, we have
	\[
		\left|\frac{\zeta(rs)}{\zeta((r-j)s) \zeta(js)} \right| > r.
	\]
\end{lemma}

\begin{proof}
	By the functional equation,
	\[
		\left| \frac{\zeta(rs)}{\zeta((r-j)s) \zeta(js)} \right| = \pi \left|\frac{\sin \left(\frac{\pi}{2} rs \right)}{\sin \left(\frac{\pi}{2} (r-j)s \right) \sin \left(\frac{\pi}{2} js \right)}\right| \left|\frac{\Gamma(1-rs)}{\Gamma(1-(r-j)s) \Gamma(1-js)}\right| \left|\frac{\zeta(1-rs)}{\zeta(1-(r-j)s) \zeta(1-js)}\right|.
	\]
Applying Lemma \ref{lemma1} and \ref{lemma2}, we get the lemma.
\end{proof}

Now we prove the following key proposition.

\begin{prop}
	Let $r \geq 2$ and $k \geq 0$. For any $s \in \partial R_k(r)$, we have $|\zeta_r(s)| < |\zeta(rs)|$.
\end{prop}

\begin{proof}
	We prove it by induction on $r$. Note that $\zeta(rs)$ is never equal to $0$ on the boundary $\partial R_k$. For $r = 2$, by Lemma \ref{lemma3} (with $r=2$) 
and \eqref{zeta_id}, we have
	\[
		|\zeta_2(s)| \leq \frac{1}{2} \bigg(|\zeta(s)|^2 + |\zeta(2s)| \bigg) \leq \frac{1}{2} \left(\frac{1}{2} |\zeta(2s)| + |\zeta(2s)| \right) < |\zeta(2s)|.
	\]
Now suppose that the proposition holds for less than $r$. Then by the induction
assumption and lemma \ref{lemma3} we have
	\begin{align*}
		|\zeta_r(s)| &\leq \frac{1}{r} \sum_{j=1}^r |\zeta_{r-j}(s)| |\zeta(js)| < \frac{1}{r} |\zeta(rs)| + \frac{1}{r} \sum_{j=1}^{r-1} |\zeta((r-j)s)| |\zeta(js)|\\
			&\leq \frac{1}{r} |\zeta(rs)| + \frac{r-1}{r^2} |\zeta(rs)| < |\zeta(rs)|,
	\end{align*}
	which finishes the proof.
\end{proof}

We finally provide a proof of Theorem \ref{Th-Toshiki}.   By (\ref{zeta_id}) and the above proposition, for $s \in \partial R_k(r)$,
\begin{align}\label{Rouche}
	\left| \sum_{j=1}^{r-1} (-1)^{j-1} \zeta_{r-j}(s) \zeta(js) \right| &\leq \sum_{j=1}^{r-1} |\zeta_{r-j}(s)| |\zeta(js)| < \sum_{j=1}^{r-1} |\zeta((r-j)s)| |\zeta(js)| \nonumber\\
		&< \frac{r-1}{r} |\zeta(rs)| < |\zeta(rs)|.
\end{align}
Therefore, by Rouch\'{e}'s theorem, $\zeta(rs)$ and $\zeta_r(s)$ have the same number of zeros in the interior of each $R_k(r)$.
Since $\zeta(rs)$ has exactly one zero in the interior of each $R_k(r)$ ($k>0$), 
$\zeta_r(s)$ has the unique zero in the interior of $R_k$ for $k > 0$. Since $\zeta_r(s)$ is symmetric with respect to the real-axis, the zero is real. By taking $\varepsilon \to 0^+$, we also see that there is no zero in the interior of $R_0(r)$.

By replacing $R_k(r)$ $(k>0)$ with a slightly larger rectangle such that (\ref{Rouche}) also holds, we can check that there is no zero of $\zeta_r(s)$ on the end points of $R_k(r) \cap \mathbb{R}$. Thus we complete the proof of the theorem.


%
%

%
%


\

\end{document}